\documentclass[10pt,leqno,oneside]{amsart}
\usepackage{amsmath,amssymb,amsxtra,comment,graphicx,psfrag}
\usepackage{bm,mathrsfs}
\usepackage{mathtools}
\usepackage{xspace}
\usepackage{color}
\usepackage{enumerate}
\usepackage{pdfsync}
\usepackage{hhline}
\usepackage[nohead,margin=1.1in]{geometry}
\usepackage{diagbox}

\usepackage{fancyhdr}

\allowdisplaybreaks

\newcommand{\al}{\alpha}

\def\Dal{{\partial_t^\al}}
\def\bDal{{\bar\partial_\tau^\al}}
\def\d{{\rm d}}

\theoremstyle{plain}
\newtheorem{theorem}{Theorem}[section]
\newtheorem{remark}{Remark}[section]

\newtheorem{lemma}{Lemma}[section]

\newtheorem{corollary}{Corollary}[section]
\numberwithin{equation}{section}

\newcommand{\R}{\mathbb{R}}

\def\d{{\rm d}}

\begin{document}

\title[]{Subdiffusion with a time-dependent coefficient: analysis and numerical solution}

\author[Bangti Jin]{Bangti Jin}
\address{Department of Computer Science, University College London, Gower Street, London, WC1E 2BT, UK.}
\email {b.jin@ucl.ac.uk,bangti.jin@gmail.com}

\author[Buyang Li]{$\,\,$Buyang Li$\,$}
\address{Department of Applied Mathematics,
The Hong Kong Polytechnic University, Kowloon, Hong Kong}
\email {bygli@polyu.edu.hk}

\author[Zhi Zhou]{$\,\,$Zhi Zhou$\,$}
\address{Department of Applied Mathematics,
The Hong Kong Polytechnic University, Kowloon, Hong Kong}
\email {zhizhou@polyu.edu.hk}


\begin{abstract}
In this work, a complete error analysis is presented for fully discrete solutions of the subdiffusion
equation with a time-dependent diffusion coefficient, obtained by the Galerkin finite element method
with conforming piecewise linear finite elements in space
and backward Euler convolution quadrature in time. The regularity of the solutions of the subdiffusion
model is proved for both nonsmooth initial data and incompatible source term. Optimal-order convergence of the
numerical solutions is established using the proven solution regularity and a novel perturbation argument
via freezing the diffusion coefficient at a fixed time. The analysis is supported by numerical experiments.
\end{abstract}

\maketitle

\section{Introduction}\label{sec:intro}
Let $\Omega\subset\mathbb{R}^d $ ($d\ge 1$) be a convex polyhedral domain with a boundary $\partial\Omega$.
Consider the following fractional-order parabolic problem for the function $u(x,t)$:
\begin{align}\label{eqn:pde}
\left\{\begin{aligned}
&\Dal u(x,t) - \nabla\cdot(a(x,t)\nabla u(x,t)) = f(x,t) &&(x,t)\in\Omega\times(0,T], \\
&u(x,t)=0  &&(x,t)\in \partial\Omega\times(0,T], \\
&u(x,0)=u_0(x) &&x\in\Omega,
\end{aligned}
\right.
\end{align}
where $T>0$ is a fixed final time, $f \in L^\infty(0,T;L^2(\Omega))$ and $u_0\in L^2(\Omega)$ are given
source term and initial data, respectively, and $a(x,t)\in\R^{d\times d}$ is a symmetric matrix-valued
diffusion coefficient such that for some constant $\lambda\geq 1$
\begin{align}
  &\lambda^{-1}|\xi|^2\le a(x,t)\xi\cdot\xi \le \lambda|\xi|^2, && \forall\, \xi\in\R^d, \,\, \forall\, (x,t)\in \Omega\times(0,T], \label{Cond-1} \\
  &|\partial_t a(x,t)|+|\nabla_xa(x,t)|+|\nabla_x\partial_t a(x,t)|\le c, &&\forall\, (x,t)\in \Omega\times(0,T]. \label{Cond-2}
\end{align}
The notation $\Dal u(t)$ denotes the Caputo derivative in time of order $\alpha\in(0,1)$, defined
by \cite[p. 70]{KilbasSrivastavaTrujillo:2006}
\begin{align}\label{eqn:RLderive}
   \Dal u(x,t)= \frac{1}{\Gamma(1-\alpha)}\int_0^t(t-s)^{-\alpha}\partial_s u (x,s)\d s .
\end{align}

The literature on the numerical analysis of the subdiffusion problem is vast; see \cite{LinXu:2007,
JinLazarovZhou:SIAM2013,JinLazarovPasciakZhou:2015,KaraaMustapha:2017} for a rather incomplete list and
the overview \cite{JinLazarovZhou:2018review} (and the references therein). The work
\cite{JinLazarovZhou:SIAM2013} analyzed two spatially semidiscrete schemes, i.e., Galerkin finite element
method (FEM) and lumped mass method, and derived nearly optimal order error estimates for the homogeneous
problem. The inhomogeneous case was analyzed in \cite{JinLazarovPasciakZhou:2015}. See
\cite{KaraaMustapha:2017} for a finite volume element discretization, and \cite{Karaa:2017} for a unified
approach. There are a number of fully discrete schemes, e.g., convolution quadrature
\cite{YusteAcedo:2005,JinLazarovZhou:SISC2016}, piecewise polynomial interpolation
\cite{SunWu:2006,LinXu:2007,Alikhanov:2015,GaoSunZhang:2014,YanKhanFord:2018,StynesORiordanGracia:2017}, discontinuous Galerkin method
\cite{McLeanMustapha:2009,McLeanMustapha:2015}; and some of them have an $O(\tau)$ rate for
nonsmooth data, with $\tau$ being time step size. However, all these works analyzed
only the case that the diffusion coefficient $a$ is independent of the time $t$. These works mostly employ
Laplace transform and its discrete analogue for analysis, which are not directly applicable to the case of
a time-dependent coefficient. Recently, Mustapha \cite{Mustapha:2017} analyzed the spatially semidiscrete
Galerkin FEM for \eqref{eqn:pde} using a novel energy argument, and proved optimal-order convergence rates
for both smooth and nonsmooth initial data (with a zero source term) based on certain assumptions on the
regularity of the PDE's solution.

In this article, using a novel perturbation argument, we present a new approach to analyze a fully discrete
scheme for problem \eqref{eqn:pde} based on the Galerkin FEM in space and backward Euler (BE) convolution
quadrature in time, covering initial data and source term simultaneously. The main contributions of this paper
are as follows. First, we give a complete existence, uniqueness and regularity theory for problem \eqref{eqn:pde}
in Theorems \ref{thm:reg-space1}--\ref{thm:reg-time}, which are crucial to the error analysis.
Second, we derive sharp error estimates for the spatially semidiscrete Galerkin FEM. This is achieved by
combining error estimates for a time-independent coefficient and a perturbation argument in time. Third,
we derive nearly sharp error estimates for the fully discrete method. All error estimates are given
directly in terms of the regularity of the initial data and source term, under mild regularity assumptions
on the diffusion coefficient $a(x,t)$ that are weaker than the assumptions in \cite{Mustapha:2017}; see
Remark \ref{rmk:reg} for the precise statement.

There are a few relevant works on standard parabolic problems with a time-dependent coefficient
\cite{LuskinRannacher:1982,Sammon:1983,Savare:1993,LeeLeeSheen:2013}. For example, Luskin and
Rannacher \cite{LuskinRannacher:1982} proved optimal order error estimates for both spatially
semidiscrete and fully discrete problems (by BE method) using a novel energy argument, and Sammon
\cite{Sammon:1983} analyzed fully discrete schemes with linear multistep methods. Our error analysis relies crucially
on a perturbation argument, using basic estimates given in Lemmas \ref{lemma-fem-1} and \ref{lemma-fem-2},
which are of independent interest. Generally, the idea of freezing coefficients and perturbation in
time has been proved very useful in combination with energy estimates \cite{Savare:1993} and $L^p$
estimates \cite{AkrivisLiLubich2017,KunstmannLiLubich2017,LiSun2015-Regularity}. In this work, we
have successfully adapted the idea to the subdiffuion model.

The rest of the paper is organized as follows. In Section \ref{sec:reg}, we discuss temporal
and spatial regularity of the solution for nonsmooth problem data. Then in Section \ref{sec:semi},
we prove optimal-order convergence of the spatially semidiscrete Galerkin FEM for both homogeneous
and inhomogeneous problems. In Section \ref{sec:fully}, we present the error analysis for the fully discrete
FEM and prove first-order convergence in time. Last, in Section \ref{sec:numerics}, we present
numerical examples to support the theoretical analysis. Throughout, the notation $c$, with or without
a subscript, denotes a generic positive constant, which may differ at each occurrence, but is always
independent of the mesh size $h$ and step size $\tau$.

\section{Regularity theory}\label{sec:reg}
In this section we investigate the regularity of the solutions of problem \eqref{eqn:pde}. For any
function $f(x,t)$ defined on $\Omega\times(0,T)$, we denote by $f(t)$ the function $f(\cdot,t)$.
Let $-\Delta:H^1_0(\Omega)\cap H^2(\Omega)\rightarrow L^2(\Omega)$ be the negative Laplacian operator
with a zero Dirichlet boundary condition, and $\{(\lambda_j,\varphi_j)\}$ be its eigenvalues
ordered nondecreasingly (with multiplicity counted) and the corresponding eigenfunctions normalized
in the $L^2(\Omega)$ norm. For any $r\geq0$, we denote the space $\dot H^r(\Omega)=\{v\in L^2(\Omega): (-
\Delta)^\frac{r}{2}v\in L^2(\Omega)\}$, with the norm \cite[Chapter 3]{Thomee:2006}
\begin{equation*}
  \|v\|_{\dot H^r(\Omega)}^2=\sum_{j=1}^\infty\lambda_j^r(v,\varphi_j)^2.
\end{equation*}
Then we have $\dot H^0(\Omega)= L^2(\Omega)$,
$\dot H^1(\Omega)=H_0^1(\Omega)$, and $\dot H^2(\Omega)= H^2(\Omega)\cap H_0^1(\Omega)$.

\subsection{Subdiffusion with a time-independent coefficient}
First we recall basic properties of the subdiffusion model with a time-independent
diffusion coefficient, i.e., $a(x,t)=a(x)$. Accordingly, we denote by $A:H^1_0(\Omega)\cap H^2(\Omega)
\rightarrow L^2(\Omega)$ an elliptic operator, defined by
$$A\phi (x):=-\nabla\cdot(a(x)\nabla \phi(x)) , $$
and consider the problem
\begin{equation}\label{PDE-independent}
\Dal u(t) + Au(t) = f(t)\ \,\,\,t\in(0,T],
\quad \mbox{with }u(0)=u_0.
\end{equation}
This problem has been studied in \cite{Bajlekov:2001,EidelmanKochubei:2004,Luchko:2009,McLean:2010,SakamotoYamamoto:2011,JinLiZhou:nonlinear}.
The following maximal $L^p$-regularity holds \cite{Bajlekov:2001}.
\begin{lemma} \label{Lemma:indp}
If $u_0=0$ and $f\in L^p(0,T;L^2(\Omega))$ with $1<p<\infty$, then problem \eqref{PDE-independent}
has a unique solution $u\in L^p(0,T;\dot H^2(\Omega))$ such that $\Dal u\in L^p(0,T;L^2(\Omega))$ and
\begin{equation*}
\|u\|_{L^p(0,T;\dot H^2(\Omega))}+\|\Dal u\|_{L^p(0,T;L^2(\Omega))}\le c\|f\|_{L^p(0,T;L^2(\Omega))},
\end{equation*}
where the constant $c$ does not depend on $f$ and $T$.
\end{lemma}

By means of Laplace transform,
the solution $u(t)$ can be represented by \cite[Section 4]{JinLiZhou:nonlinear}
\begin{align}\label{eqn:Sol-expr-u-const}
u(t)= F(t)u_0 + \int_0^t E(t-s) f(s) \d s ,
\end{align}
where the solution operators $F(t)$ and $E(t)$ are defined by
\begin{align}
&F(t):=\frac{1}{2\pi {\rm i}}\int_{\Gamma_{\theta,\delta }}e^{zt} z^{\alpha-1} (z^\alpha+A )^{-1}\, \d z ,
\label{eqn:F} \\
&E(t):=\frac{1}{2\pi {\rm i}}\int_{\Gamma_{\theta,\delta}}e^{zt}  (z^\alpha+A)^{-1}\, \d z ,
\label{eqn:E}
\end{align}
with integration over a contour $\Gamma_{\theta,\delta}\subset\mathbb{C}$ (oriented with an increasing imaginary part):
\begin{equation*}
  \Gamma_{\theta,\delta}=\left\{z\in \mathbb{C}: |z|=\delta, |\arg z|\le \theta\right\}\cup
  \{z\in \mathbb{C}: z=\rho e^{\pm\mathrm{i}\theta}, \rho\ge \delta\} .
\end{equation*}
Throughout, we fix $\theta \in(\frac{\pi}{2},\pi)$ so that $z^{\al} \in \Sigma_{\al\theta}\subset
\Sigma_{\theta}:=\{0\neq z\in\mathbb{C}: {\rm arg}(z)\leq\theta\},$ for all $z\in\Sigma_{\theta}$.
The next lemma gives smoothing properties of $F(t)$ and $E(t)$, which follow from
the resolvent estimate
\begin{equation} \label{eqn:resol}
  \| (z+A)^{-1} \|\le c_\phi |z|^{-1},  \quad \forall z \in \Sigma_{\phi},
  \,\,\,\forall\,\phi\in(0,\pi) ,
\end{equation}
where $\|\cdot\|$ denotes the operator norm from $L^2(\Omega)$ to $L^2(\Omega)$.
\begin{lemma}\label{lem:smoothing}
The operators $F$ and $E$ defined in \eqref{eqn:F} and \eqref{eqn:E} satisfy the
following properties.
\begin{itemize}
\item[$\rm(i)$] $t^{-\alpha}\|A^{-1}(F(t)-I)\|+\|F(t)-I\|\le c,\quad\forall\,t\in(0,T]$; 
\item[$\rm(ii)$] $t^{1-\alpha}\|E(t)\|+t^{2-\alpha}\|E'(t)\|+t\|A  E(t)\|\le c , \quad\forall\,t\in(0,T]$;
\item[$\rm(iii)$] $t^\alpha \|AF(t)\| + t^{1-\beta\alpha}\|A^{-\beta}F'(t)\|\leq c ,\quad\forall\,t\in(0,T], \beta\in[0,1]$.
\end{itemize}
\end{lemma}
\begin{proof}
Parts (i) and (ii) have been shown in \cite[Lemma 3.4]{JinLiZhou:nonlinear}.
By letting $\delta=t^{-1}$ in $\Gamma_{\theta,\delta}$ and $z=s\cos\varphi+{\rm i}s\sin\varphi$, using \eqref{eqn:resol}, we have
(with $|\d z|$ being the arc length of $\Gamma_{\theta,\delta}$)
\begin{align*}
\|AF(t)\|
&=\bigg\|\frac{1}{2\pi {\rm i}}\int_{\Gamma_{\theta,\delta }}e^{zt}z^{\alpha-1} A (z^\alpha+A )^{-1}\, \d z\bigg\| \le c\int_{\Gamma_{\theta,\delta }}
e^{{\rm \Re}(z)t} |z|^{\alpha-1 } \, |\d z| \\
&\le c \int_{\delta}^\infty e^{st\cos\theta} s^{\alpha-1} \d s
+ c\int_{-\theta}^\theta e^{\cos\varphi} \delta^{ \alpha} \d \varphi   \le ct^{- \alpha } .
\end{align*}
Next, direct computation gives $F'(t) = \frac{1}{2\pi\rm i}\int_{\Gamma_{\theta,\delta}}e^{zt}z^\alpha (z^\alpha+A)\,\d z$, and
thus by the estimate \eqref{eqn:resol},
\begin{equation*}
  \|F'(t)\| \leq c\int_{\Gamma_{\theta,\delta }}
e^{{\rm \Re}(z)t} |z|^\alpha |z|^{-\alpha}\, |\d z|  \le ct^{-1},
\end{equation*}
which shows the assertion for $\beta=0$.
Meanwhile, by the identity $z^\alpha (z^\alpha+A)^{-1}=I-A(z^\alpha+A)^{-1}$, we have $F'(t) = \frac{1}{2\pi\rm i}
\int_{\Gamma_{\theta,\delta}}e^{zt}z^\alpha (z^\alpha+A)\,\d z = -\frac{1}{2\pi\rm i}
\int_{\Gamma_{\theta,\delta}}e^{zt}A(z^\alpha+A)\,\d z$, and thus
\begin{equation*}
  \|A^{-1}F'(t)\| \leq c\int_{\Gamma_{\theta,\delta }}
e^{{\rm \Re}(z)t} |z|^{-\alpha}\, |\d z|  \le ct^{-1+\alpha}.
\end{equation*}
This shows the assertion for $\beta=1$. Then the desired bound on $t^{1-\beta\alpha}\|A^{-\beta}F'(t)\|$ in part (iii) follows by interpolation.
\end{proof}

\subsection{Regularity theory for subdiffusion with a time-dependent coefficient}

Now we develop the regularity theory for the case of a time-dependent diffusion coefficient.
The work \cite{Zacher:2013} gave some interior H\"{o}lder estimates for bounded measurable coefficients.
Recently, Kubica et al \cite{KubicaYamamoto:2017}
showed the unique existence by approximating the coefficients by smooth functions,
and derived several regularity estimates.
We shall provide a different approach to derive
regularity estimates in Sobolev spaces, which are essential for the error
analysis in Sections \ref{sec:semi} and \ref{sec:fully}.

We define a time-dependent elliptic operator $A(t): \dot H^2(\Omega)\rightarrow L^2(\Omega)$ by
$$
A(t)\phi=- \nabla\cdot(a(x,t)\nabla \phi),\quad\forall \phi\in \dot H^2(\Omega).
$$
Under condition \eqref{Cond-2}, the following estimate holds:
\begin{equation}\label{eqn:basic-est}
  \|(A(t)-A(s))v\|_{L^2(\Omega)}\leq c|t-s|\|v\|_{H^2(\Omega)}.
\end{equation}

First we give the existence, uniqueness and regularity of solutions to problem \eqref{eqn:pde} with $u_0=0$.
\begin{theorem}\label{thm:reg-space1}
Under conditions \eqref{Cond-1}-\eqref{Cond-2}, with $u_0=0$ and $f\in L^p(0,T;L^2(\Omega))$, $1/\alpha<p<\infty$,
problem \eqref{eqn:pde} has a unique solution $u\in C([0,T];L^2(\Omega))\cap L^p(0,T;\dot H^2(\Omega))$
such that $\Dal u\in L^p(0,T;L^2(\Omega))$.
\end{theorem}
\begin{proof}
For any $\theta\in[0,1]$, consider the following fractional-order parabolic problem
\begin{align}\label{PDE-theta}
\begin{aligned}
&\Dal u(t) + A(\theta t)u(t)
= f(t) ,\quad t\in(0,T],\quad \mbox{with }
u(0)=0,
\end{aligned} 
\end{align}
and define a set
$$D=\{\theta\in[0,1]:\mbox{ \eqref{PDE-theta} has a solution
$u\in L^p(0,T;\dot  H^2(\Omega))$ such that $\Dal u\in L^p(0,T;L^2(\Omega))$} \}.$$
Lemma \ref{Lemma:indp} implies $0\in D$ and so $D\neq\emptyset$.

For any $\theta\in D$, by rewriting \eqref{PDE-theta} as
\begin{align}\label{PDE-theta1}
\begin{aligned}
&\Dal u(t) + A(\theta t_0)u(t) = f(t) + (A(\theta t_0)-A(\theta t))u(t) ,\quad t\in(0,T],\quad \mbox{with }
u(0)=0 ,
\end{aligned}
\end{align}
and by applying Lemma \ref{Lemma:indp} in the time interval $(0,t_0)$, we obtain
\begin{align}\label{maxLp-1}
&\|\Dal u\|_{L^p(0,t_0;L^2(\Omega))}+\| u\|_{L^p(0,t_0;H^2(\Omega))} \nonumber \\
\le&c\|f\|_{L^p(0,t_0;L^2(\Omega))}
+ c\|(A(\theta t_0)-A(\theta t))u(t)\|_{L^p(0,t_0;L^2(\Omega))} \nonumber \\
\le&c\|f\|_{L^p(0,t_0;L^2(\Omega))}
+ c\|(t_0-t)u(t)\|_{L^p(0,t_0;H^2(\Omega))} ,
\end{align}
where the last line follows from \eqref{eqn:basic-est}. Let $g(t)=\| u\|_{L^p(0,t;H^2(\Omega))}^p $, which satisfies
$g'(t) = \| u(t)\|_{H^2(\Omega)}^p.$ Then \eqref{maxLp-1} and integration by parts imply
\begin{align*}
g(t_0)&\le c\|f\|_{L^p(0,t_0;L^2(\Omega))}^p    + c\int_0^{t_0}(t_0-t)^p g'(t)\d t \nonumber \\
&= c\|f\|_{L^p(0,t_0;L^2(\Omega))}^p    + cp\int_0^{t_0}(t_0-t)^{p-1}g(t)\d t\nonumber\\
&\le c\|f\|_{L^p(0,t_0;L^2(\Omega))}^p + c\int_0^{t_0}g(t)\d t ,
\end{align*}
which implies (via the standard Gronwall's inequality)
\begin{align*}
g(t_0)\le c\|f\|_{L^p(0,t_0;L^2(\Omega))}^p,\quad \mbox{i.e., } \| u\|_{L^p(0,t_0;H^2(\Omega))} \leq c\|f\|_{L^p(0,t_0;L^2(\Omega))}.
\end{align*}
Substituting the last inequality into \eqref{maxLp-1} yields
\begin{align}\label{maxLp-2}
&\|\Dal u\|_{L^p(0,t_0;L^2(\Omega))}+\| u\|_{L^p(0,t_0;H^2(\Omega))} \le
c\|f\|_{L^p(0,t_0;L^2(\Omega))} .
\end{align}
Since the estimate \eqref{maxLp-2} is independent of $\theta\in D$,
 $D$ is a closed subset of $[0,1]$.

Now we show that $D$ is also open with respect to
the subset topology of $[0,1]$.  In fact, if $\theta_0\in D$, then problem \eqref{PDE-theta} can be rewritten as
\begin{align}\label{PDE-theta3}
\begin{aligned}
&\Dal u(t) + A(\theta_0 t)u(t) +(A(\theta t)-A(\theta_0 t))u(t)
= f(t) ,\quad t\in(0,T],\quad \mbox{with }
u(0)=0 ,
\end{aligned}
\end{align}
which is equivalent to
\begin{align}\label{PDE-theta4}
&\Big[1 + (\Dal + A(\theta_0 t))^{-1} (A(\theta t)-A(\theta_0 t)) \Big]u(t)
= (\Dal + A(\theta_0 t))^{-1} f(t).
\end{align}
It follows from \eqref{maxLp-2} that the operator $ (\Dal + A(\theta_0 t))^{-1} (A(\theta t)-A(\theta_0 t))$ is small in the sense that
$$
\| (\Dal + A(\theta_0 t))^{-1} (A(\theta t)-A(\theta_0 t))\|_{L^p(0,T;H^2(\Omega))\rightarrow L^p(0,T;H^2(\Omega))}
\le c|\theta-\theta_0|.
$$
Thus for $\theta$ sufficiently close to $\theta_0$, the operator $1 + (\Dal + A(\theta_0 t))^{-1} (A(\theta t)-
A(\theta_0 t)) $ is invertible on $L^p(0,T;\dot H^2(\Omega))$, which implies $\theta\in D$. Thus $D$
is open with respect to the subset topology of $[0,1]$.
Since $D$ is both closed and open respect to the subset topology of $[0,1]$, $D=[0,1]$.
Further, note that for $1/\alpha<p<\infty$, the inequality \eqref{maxLp-2} and the condition $u(0)=0$
directly imply $u\in C([0,T];L^2(\Omega))$ \cite[Theorem
2.1]{JinLiZhou:nonlinear}, which completes the proof of the theorem.
\end{proof}

The following generalized Gronwall's inequality is useful (\cite[Lemma 6.3]{ElliottLarsson:1992}
and \cite[Exercise 3, p. 190]{Henry:1981}).
\begin{lemma}\label{lem:gronwall}
Let the function $\varphi(t)\geq0$ be continuous for $0< t\leq T$. If
\begin{equation*}
  \varphi(t)\leq at^{-1+\alpha} + b\int_0^t (t-s)^{-1+\beta}\varphi(s)\d s,\quad 0<t\leq T,
\end{equation*}
for some constants $a,b\geq0$, $\alpha,\beta>0$, then there is a constant $c=c(b,T,\alpha,\beta)$ such that
\begin{equation*}
  \varphi(t)\leq cat^{-1+\alpha},\quad 0<t\leq T.
\end{equation*}
\end{lemma}

Next we give the spatial regularity of the solution $u$ for the case $f=0$.
\begin{theorem}\label{thm:reg-space2}
Under conditions \eqref{Cond-1}-\eqref{Cond-2}, with $u_0\in \dot H^\beta(\Omega)$, $0\leq \beta\le 2$, and $f=0$, problem \eqref{eqn:pde}
has a unique solution $u\in C([0,T];L^2(\Omega))\cap C((0,T];\dot H^2(\Omega))$ such that $\Dal u\in C((0,T];L^2(\Omega))$, and
\begin{align*}
 \|u(t)\|_{H^2(\Omega)}\le ct^{-(1-\beta/2)\alpha} \|u_0\|_{\dot H^\beta(\Omega)}.
\end{align*}
\end{theorem}
\begin{proof}
The existence and uniqueness of a solution can be proved in the same way as Theorem \ref{thm:reg-space1} based on
the a priori estimate below. We rewrite problem \eqref{eqn:pde} as
\begin{align*}
\begin{aligned}
&\Dal u(t) + A(t_0)u(t)
=  (A(t_0)-A(t))u(t) + f(t) ,\quad t\in(0,T],\quad\mbox{with }
u(0)=u_0 ,
\end{aligned}
\end{align*}
Then the solution $u(t)$ can be represented by
\begin{align}\label{eqn:Sol-expr-u}
u(t)
&=F(t;t_0)u_0
+ \int_0^t E(t-s; t_0) (A(t_0)-A(s))u(s) \d s+ \int_0^t E(t-s; t_0) f(s) \d s,
\end{align}
where the operators $F(t;t_0)$ and $E(t;t_0)$ are defined respectively by
\begin{align*}
F(t;t_0):=\frac{1}{2\pi {\rm i}}\int_{\Gamma_{\theta,\delta }}e^{zt} z^{\alpha-1} (z^\alpha+A(t_0) )^{-1}\, \d z
\quad \mbox{and}\quad E(t;t_0):=\frac{1}{2\pi {\rm i}}\int_{\Gamma_{\theta,\delta}}e^{zt}  (z^\alpha+A(t_0) )^{-1}\, \d z .
\end{align*}
In the case $f=0$, applying $A(t_0)$ to both sides of \eqref{eqn:Sol-expr-u} yields
\begin{align*}
A(t_0)u(t)=& A(t_0)F(t; t_0)u_0  + \int_0^t A(t_0)E(t-s; t_0)(A(t_0)-A(s))u(s) \d s.
\end{align*}
Then conditions \eqref{Cond-1}-\eqref{Cond-2} and Lemma \ref{lem:smoothing}(ii) imply
\begin{align*}
\|u(t_0)\|_{H^2(\Omega)}&\le c\|A(t_0)F(t_0; t_0)u_0 \|_{L^2(\Omega)}
+ c\int_0^{t_0} \|A(t_0)E(t_0-s; t_0)\|\|(A(t_0)-A(s))u(s)\|_{L^2(\Omega)} \d s\\
&\le ct_0^{-(1-\beta/2)\alpha}\|u_0 \|_{\dot H^\beta(\Omega)}
+ c\int_0^{t_0} (t_0-s)^{-1}(t_0-s)\|u(s)\|_{H^2(\Omega)} \d s \\
&= ct_0^{-(1-\beta/2)\alpha}\|u_0 \|_{\dot H^\beta(\Omega)}
+ c\int_0^{t_0} \|u(s)\|_{H^2(\Omega)} \d s ,
\quad\forall \, t_0\in(0,T].
\end{align*}
The desired estimate follows from the generalized Gronwall's inequality in Lemma \ref{lem:gronwall}.
It remains to show $ u\in C([0,T];L^2(\Omega))\cap C((0,T];H^2(\Omega))$. Indeed, note 
that (by fixing $t_0=0$)
\begin{equation*}
  \|u(t)-u_0\|_{L^2(\Omega)} \leq \|F(t;0)u_0-u_0\|_{L^2(\Omega)} + \int_0^t(t-s)^{-\alpha}s\|u(s)\|_{H^2(\Omega)}\d s,
\end{equation*}
which together with the bound on $\|u(s)\|_{H^2(\Omega)}$ implies
\begin{align*}
  & \lim_{t\to0^+}\|u(t)-u_0\|_{L^2(\Omega)}\\
  \leq & \lim_{t\to0^+}\|F(t;0)u_0-u_0\|_{L^2(\Omega)} + \lim_{t\to0^+}c\int_0^t(t-s)^{\alpha-1}s^{1-(1-\beta/2)\alpha}\d s\|u_0\|_{\dot H^\beta(\Omega)}=0.
\end{align*}
i.e., $\lim_{t\to0^+}u(t)=u_0$ in $L^2(\Omega)$. The rest of the assertion follows similarly.
This completes the proof.
\end{proof}

To analyze the temporal regularity, we first give three technical lemmas.
\begin{lemma}\label{lem:time-t}
Let conditions \eqref{Cond-1} and \eqref{Cond-2} be fulfilled, and $u$ be the solution to problem
\eqref{eqn:pde} with $u_0\in L^2(\Omega)$ and $f=0$. Then there holds
\begin{equation*}
  \|\frac{\d}{\d t}\int_0^t E(t-s; t_0)(A(t_0)-A(s))u(s) \d s\big|_{t=t_0}\|_{L^2(\Omega)}\leq c\|u_0\|_{L^2(\Omega)}.
\end{equation*}
\end{lemma}
\begin{proof}
Let ${\rm I}=\frac{\d}{\d t}\int_0^t E(t-s; t_0)(A(t_0)-A(s))u(s) \d s|_{t=t_0}$. Then
\begin{align}\label{eqn:ut-split}
{\rm I}&= \lim_{\varepsilon\rightarrow 0}\frac{1}{\varepsilon}
\bigg(\int_0^{t_0+\varepsilon} E(t_0+\varepsilon-s; t_0)(A(t_0)-A(s))u(s) \d s \nonumber \\
&\qquad\qquad\qquad
-\int_0^{t_0} E(t_0-s; t_0)(A(t_0)-A(s))u(s) \d s\bigg) =: \lim_{\varepsilon\rightarrow 0} \Lambda(\varepsilon).
\end{align}
If $\varepsilon>0$, then
\begin{align}\label{I-II+}
\Lambda(\varepsilon)&=\frac{1}{\varepsilon}
\int_{t_0}^{t_0+\varepsilon} E(t_0+\varepsilon-s; t_0)(A(t_0)-A(s))u(s) \d s \nonumber \\
&\quad
+\int_0^{t_0} \frac{E(t_0+\varepsilon-s; t_0)-E(t_0-s; t_0)}{\varepsilon}(A(t_0)-A(s))u(s) \d s  =: {\rm I}_++ {\rm II}_+.
\end{align}
By applying Lemma \ref{lem:smoothing}(ii), \eqref{eqn:basic-est} and Theorem \ref{thm:reg-space2}, we deduce
\begin{align*}
\|{\rm I}_+\|_{L^2(\Omega)}&\le c\varepsilon^{-1}\int_{t_0}^{t_0+\varepsilon}
\| E(t_0+\varepsilon-s; t_0)\| \|(A(t_0)-A(s))u(s)\|_{L^2(\Omega)} \d s \\
&\le c \varepsilon^{-1}\int_{t_0}^{t_0+\varepsilon}
|t_0+\varepsilon-s|^{\alpha-1} |t_0-s| \|u(s)\|_{H^2(\Omega)} \d s \\
&\le c\varepsilon^{-1}\int_{t_0}^{t_0+\varepsilon}
|t_0+\varepsilon-s|^{\alpha-1} |t_0-s| s^{-\alpha} \|u_0\|_{L^2(\Omega)} \d s\\
&\le c\varepsilon^{-1}\int_{t_0}^{t_0+\varepsilon}
|t_0+\varepsilon-s|^{\alpha-1} (\varepsilon) t_0^{-\alpha} \|u_0\|_{L^2(\Omega)} \d s
\leq c \varepsilon^{\alpha} t_0^{-\alpha} \|u_0\|_{L^2(\Omega)} ,
\end{align*}
and similarly,
\begin{align*}
\|{\rm II}_+\|_{L^2(\Omega)}
&=\bigg\|\int_0^{t_0} \int_0^1 E'(t_0+\theta\varepsilon-s; t_0)(A(t_0)-A(s))u(s) \,\d\theta \d s \bigg\|_{L^2(\Omega)} \\
& \le c \int_0^1\int_0^{t_0} (t_0+\theta\varepsilon-s)^{ \alpha-2}(t_0-s)  \| u(s) \|_{H^2(\Omega)} \,\d s\d\theta \\
&\le c \int_0^{t_0} (t_0-s)^{\alpha-1}\| u(s) \|_{H^2(\Omega)} \,\d s\\
& \le c \int_0^{t_0}(t_0-s)^{ \alpha-1} s^{- \alpha}\| u_0 \|_{L^2(\Omega)}  \,\d s \le c \| u_0 \|_{L^2(\Omega)}.
\end{align*}
If $-t_0<\varepsilon<0$, then
\begin{align*}
\Lambda(\varepsilon)&=\varepsilon^{-1}
\int_{t_0-|\varepsilon|}^{t_0} E(t_0-s; t_0)(A(t_0)-A(s))u(s) \d s \nonumber \\
&\quad+\int_0^{t_0-|\varepsilon|} \frac{E(t_0-s; t_0)-E(t_0-|\varepsilon|-s; t_0)}{\varepsilon}(A(t_0)-A(s))u(s) \d s  =: {\rm I}_-+ {\rm II}_- ,
\end{align*}
and similarly, we obtain
\begin{align*}
\|{\rm I}_-\|_{L^2(\Omega)}
\le c\varepsilon^{\alpha} (t_0+\varepsilon)^{-\alpha} \|u_0\|_{L^2(\Omega)} \quad\mbox{and}\quad
\|{\rm II}_-\|_{L^2(\Omega)}
\le c \| u_0 \|_{L^2(\Omega)}.
\end{align*}
Combining the preceding estimates yields the assertion.
\end{proof}

\begin{lemma}\label{lem:u-int}
Let conditions \eqref{Cond-1} and \eqref{Cond-2} be fulfilled, and $u$ be the solution to problem \eqref{eqn:pde} with $f\in C([0,T];L^2(\Omega))$, $\int_0^{t}(t-s)^{\alpha-1}\|f'(s)\|_{L^2(\Omega)}\d s<\infty$ and $u_0=0$. Then there holds
\begin{equation*}
  \int_{0}^{t} (t-s)^{\alpha-1}\|u(s)\|_{H^2(\Omega)}\d s \leq c \|f(0)\|_{L^2(\Omega)}+c\int_0^{t}(t-s)^{\alpha-1}\|f'(s)\|_{L^2(\Omega)}\d s.
\end{equation*}
\end{lemma}
\begin{proof}
By the solution representation \eqref{eqn:Sol-expr-u} with $u_0=0$, we have
\begin{equation*}
  A(t_0) u(t_0) = \int_0^{t_0}A(t_0)E(t_0-s;t_0)f(s) \d s + \int_0^{t_0}A(t_0)E(t_0-s;t_0)(A(s)-A(t_0))u(s)\d s: = {\rm I}+ {\rm II}.
\end{equation*}
It follows directly from the definition of the operators $E(s;t_0)$ and $F(s;t_0)$ that
the identity $A(t_0)E(s;t_0)=-\frac{\d}{\d s}F(s;t_0)=\frac{\d}{\d s}(I-F(s;t_0))$ holds.
So upon changing variables and integration by parts, we obtain
\begin{align*}
  {\rm I} & = \int_0^{t_0}A(t_0)E(s;t_0)f(t_0-s) \d s
       = \int_0^{t_0}\frac{\d}{\d s}(I-F(s;t_0))f(t_0-s) \d s\\
     & = (F(t_0;t_0)-I)f(0) - \int_0^{t_0}(I-F(s;t_0))\frac{\d}{\d s}f(t_0-s)\d s,
\end{align*}
where we have used the identity $F(0;t_0)=I$. Thus, by Lemma \ref{lem:smoothing}(i), we obtain
\begin{equation*}
  \|{\rm I}\|_{L^2(\Omega)}\leq c\|f(0)\|_{L^2(\Omega)} + c\int_0^{t_0}\|f'(s)\|_{L^2(\Omega)}\d s.
\end{equation*}
Similarly, by Lemma \ref{lem:smoothing}(ii) and \eqref{eqn:basic-est}, for the term ${\rm II}$, we have
\begin{align*}
   \|{\rm II}\|_{L^2(\Omega)}   \leq c\int_0^{t_0}(t_0-s)^{-1}|t_0-s|\|A(t_0)u(s)\|_{L^2(\Omega)}\d s =  c\int_0^{t_0}\|A(t_0)u(s)\|_{L^2(\Omega)}\d s.
\end{align*}
Let $g(t)=\int_0^t(t-s)^{\alpha-1}\|u(s)\|_{H^2(\Omega)} \d s$. Then the last two estimates together give
\begin{align*}
   g(t) & \leq c \int_0^t(t-s)^{\alpha-1}(\|{\rm I}\|_{L^2(\Omega)} + \|{\rm II}\|_{L^2(\Omega)})\,\d s  \\
       & \leq c\int_0^t(t-s)^{\alpha-1}\big(\|f(0)\|_{L^2(\Omega)}+\int_0^s\|f'(\xi)\|_{L^2(\Omega)}\d \xi+\int_0^{\xi}\|u(\xi)\|_{H^2(\Omega)}\d \xi\big)\,\d s\\
      & \leq ct^\alpha\|f(0)\|_{L^2(\Omega)} + c\int_0^t(t-s)^{\alpha}\|f'(s)\|_{L^2(\Omega)}\d s+ c\int_0^tg(s)\,\d s,
\end{align*}
where the last line follows directly from the semigroup property of Riemann-Liouville integral and change of integration orders.
Now Gronwall's inequality gives
\begin{equation*}
  g(t) \leq ct^\alpha\|f(0)\|_{L^2(\Omega)} + c\int_0^t(t-s)^{\alpha}\|f'(s)\|_{L^2(\Omega)}\d s,
\end{equation*}
from which the desired assertion follows directly.
\end{proof}

\begin{lemma}\label{lem:time-v2}
Let conditions \eqref{Cond-1} and \eqref{Cond-2} be fulfilled, and $u$ be the solution to problem \eqref{eqn:pde} with $f\in C([0,T];L^2(\Omega))$, $\int_0^{t}(t-s)^{\alpha-1}\|f'(s)\|_{L^2(\Omega)}\d s<\infty$ and $u_0=0$. Then there holds
\begin{equation*}
  \|\frac{\d}{\d t}\int_0^t E(t-s; t_0)(A(t_0)-A(s))u(s) \d s\big|_{t=t_0}\|_{L^2(\Omega)}\leq c \|f(0)\|_{L^2(\Omega)}+c\int_0^{t_0}(t_0-s)^{\alpha-1}\|f'(s)\|_{L^2(\Omega)}\d s.
\end{equation*}
\end{lemma}
\begin{proof}
For any small $\varepsilon>0$, we employ the splitting \eqref{eqn:ut-split}.
By Lemma \ref{lem:smoothing}(ii) and \eqref{eqn:basic-est}, we bound the term ${\rm I}_+$ by
\begin{align*}
\|{\rm I}_+\|_{L^2(\Omega)} &\le \varepsilon^{-1}\int_{t_0}^{t_0+\varepsilon}
\| E(t_0+\varepsilon-s; t_0)\| \|(A(t_0)-A(s)) u(s)\|_{L^2(\Omega)} \d s \\
&\le c \varepsilon^{-1} \int_{t_0}^{t_0+\varepsilon}
|t_0+\varepsilon-s|^{\alpha-1} |t_0-s| \|u(s)\|_{H^2(\Omega)} \d s\\
&\le c\int_{t_0}^{t_0+\varepsilon}
|t_0+\varepsilon-s|^{\alpha-1} \|u(s)\|_{H^2(\Omega)} \d s.
\end{align*}
This estimate, H\"{o}lder's inequality and Theorem \ref{thm:reg-space1} directly
imply $\lim_{\varepsilon\to0^+}\|{\rm I}_+\|_{L^2(\Omega)}=0$.
For the term ${\rm II}_+$,  Lemma \ref{lem:smoothing}(ii) gives
\begin{align*}
\|{\rm II}_+\|_{L^2(\Omega)}&=
\bigg\|\int_0^{t_0} \int_0^1 E'(t_0+\theta\varepsilon-s; t_0)(A(t_0)-A(s))u(s) \,\d\theta \d s \bigg\|_{L^2(\Omega)} \\
& \le c \int_0^1\int_0^{t_0} (t_0+\theta\varepsilon-s)^{ \alpha-2}(t_0-s)  \| u(s) \|_{H^2(\Omega)} \,\d s\d\theta \\
&\le c \int_0^{t_0} (t_0-s)^{\alpha-1}\| u(s) \|_{H^2(\Omega)} \,\d s,
\end{align*}
which together with Lemma \ref{lem:u-int} yields the assertion for $\varepsilon>0$.
Similar estimates hold for the case $\varepsilon<0$, and this completes the
proof of the lemma.
\end{proof}

\begin{remark}
Note that the bound on ${\rm II}_+$ in Lemma \ref{lem:time-t} blows up for $ \alpha\to 1^-$:
\begin{equation*}
  \int_0^{t_0}(t_0-s)^{ \alpha-1} s^{- \alpha} \,\d s = B(\alpha,1-\alpha),
\end{equation*}
and in view of the asymptotics $B(\alpha,1-\alpha)=O((1-\alpha)^{-1})$ as $\alpha\to 1^-$, it blows up at a rate
$1/(1-\alpha)$. Actually, this can be avoided by the following alternative argument:
\begin{align*}
\|{\rm II}_+\|_{L^2(\Omega)}
&=\bigg\|\int_0^{t_0} \int_0^1 E'(t_0+\theta\varepsilon-s; t_0)A(t_0)^{1/2}A(t_0)^{1/2}(1-A(t_0)^{-1}A(s))u(s) \,\d\theta \d s \bigg\|_{L^2(\Omega)} \\
& \le c \int_0^1\int_0^{t_0} (t_0+\theta\varepsilon-s)^{\alpha/2-2}(t_0-s)  \|A(t_0)^{1/2}u(s)\|_{L^2(\Omega)} \,\d s\d\theta \\
&\le c \int_0^{t_0} (t_0-s)^{\alpha/2-1}\| u(s) \|_{H^1(\Omega)} \,\d s\\
& \le c \int_0^{t_0}(t_0-s)^{ \alpha/2-1} s^{- \alpha/2}\| u_0 \|_{L^2(\Omega)}  \,\d s \le c \| u_0 \|_{L^2(\Omega)},
\end{align*}
where the first inequality is due to \eqref{eqn:basic-est}, Corollary \ref{cor:basic-est} below and interpolation.
The same argument can be applied to the term ${\rm II_+}$ in Lemma \ref{lem:time-v2}. Thus, the involved constants
are bounded for $\alpha\to 1^-$.
\end{remark}

Now we can give the temporal regularity of the solution $u$.
\begin{theorem}\label{thm:reg-time}
Let conditions \eqref{Cond-1}-\eqref{Cond-2} be fulfilled, and $u$ be the solution to problem \eqref{eqn:pde}.
\begin{itemize}
\item[(i)] For $ u_0\in \dot H^\beta(\Omega)$, $0\leq \beta\le 2$, and $f=0$, then
\begin{equation*}
    \| u'(t) \|_{L^2(\Omega)} \le c t^{-(1-\alpha\beta/2)} \| u_0  \|_{\dot H^\beta(\Omega)}.
\end{equation*}
\item[(ii)] For $u_0=0$, $f\in C([0,T]; L^2(\Omega))$ and $\int_0^t  (t-s)^{\alpha-1}\| f'(s) \|_{L^2(\Omega)}\,\d s<\infty$, then
\begin{align*}
   & \| u '(t) \|_{L^2(\Omega)} \le c t^{-(1-\alpha)}  \| f(0) \|_{L^2(\Omega)} +c\int_0^t (t-s)^{\alpha-1} \| f'(s)\|_{L^2(\Omega)} \,\d s.
\end{align*}
\item[(iii)] For $u_0=0$ and $f\in L^p(0,T;L^2(\Omega))$ with ${2}/{\alpha}<p<\infty$, then
\begin{equation*}
  \|u(t)\|_{H^1(\Omega)}\leq c\|f\|_{L^p(0,t;L^2(\Omega))}.
\end{equation*}
\end{itemize}
\end{theorem}
\begin{proof}
The proof employs the solution representation \eqref{eqn:Sol-expr-u}.
By Lemma \ref{lem:smoothing}(iii), we have
\begin{align*}
\big\|\frac{\d}{\d t}F(t; t_0)u_0\big|_{t=t_0}\big\|_{L^2(\Omega)}
\le ct_0^{-(1-\alpha\beta/2)}\|u_0 \|_{\dot H^\beta(\Omega)} .
\end{align*}
This and Lemma \ref{lem:time-t} yield the assertion in part (i).

To show part (ii), differentiating \eqref{eqn:Sol-expr-u} with respect to $t$ yields
\begin{align}\label{utaaa}
{u'(t_0)}
&=\frac{\d}{\d t}\int_0^t E(t-s; t_0) (A(t_0)-A(s))u(s) \d s\bigg|_{t=t_0}
+ \frac{\d}{\d t}\int_0^t E(s; t_0) f(t-s) \d s{ \bigg|_{t=t_0} }.
\end{align}
In view of the identity
$$
\frac{\d}{\d t}\int_0^t E(s; t_0) f(t-s) \d s  = E(t;t_0)f(0) + \int_0^t E(s;t_0) f'(t-s)\d s,
$$
by Lemma \ref{lem:smoothing}(ii), we have
\begin{align}
\bigg\|\frac{\d}{\d t}\int_0^t E(s; t_0) f(t-s) \d s\bigg\|_{L^2(\Omega)} &
\le \|E(t;t_0)f(0)\|_{L^2(\Omega)} + \int_0^t\| E(s;t_0) f'(t-s)\|_{L^2(\Omega)}\d s,\nonumber\\
& \leq ct^{-(1-\alpha)} \|  f(0) \|_{L^2(\Omega)} +c\int_0^t s^{\alpha-1} \| f'(t-s)\|_{L^2(\Omega)} \,\d s.\label{aaa2}
\end{align}
This and Lemma \ref{lem:time-v2} complete the proof of part (ii).

Last, for the choice ${2}/{\alpha}<p<\infty$, Lemma \ref{Lemma:indp} implies
\begin{align*}
u\in L^p(0,T;H^2(\Omega))&\cap W^{\alpha,p}(0,T;L^2(\Omega))
\hookrightarrow W^{{\alpha}/{2},p}(0,T;(L^2(\Omega),H^2(\Omega))_{1/ 2}) \\
&=W^{{\alpha}/{2},p}(0,T;H^1(\Omega)) \hookrightarrow C([0,T];H^1(\Omega)) ,
\end{align*}
where $(L^2(\Omega),H^2(\Omega))_{1/2}$ denotes the complex interpolation space between $L^2(\Omega)$
and $H^2(\Omega)$, and the last embedding is a consequence of \cite[equation (2.3)]{JinLiZhou:nonlinear}.
Then the proof of Theorem \ref{thm:reg-time} is complete.
\end{proof}

\begin{remark}\label{rmk:reg}
In the error analysis, the work \cite{Mustapha:2017} requires the following conditions on the coefficient $a(x,t)$:
$a(x,t),\partial_ta(x,t)\in L^\infty(0,T;W^{1,\infty}(\Omega))$ and $\partial_{tt}^2a(x,t)\in L^\infty(0,T;L^\infty(\Omega))$,
which are more stringent than \eqref{Cond-2}. Further, the work \cite{Mustapha:2017}
has assumed the following regularity on the solution $u$ to the homogeneous problem: for $ 0\leq p \leq q \leq 2$,
\begin{equation*}
  \|u(t)\|_{\dot H^q(\Omega)} + t\|u'(t)\|_{\dot H^q(\Omega)} \leq c t^{-(q-p)\alpha /2}\|u_0\|_{\dot H^p(\Omega)}.
\end{equation*}
In contrast, for the homogeneous problem, we proved the following estimates
under assumption \eqref{Cond-2}:
\begin{equation*}
 t^{(1-\beta/2)\alpha}\|u(t)\|_{H^2(\Omega)} + t^{1-\alpha\beta/2}\|u'(t)\|_{L^2(\Omega)} \le c\|u_0\|_{\dot H^\beta(\Omega)}
\end{equation*}
and similar estimates for the inhomogeneous problem. It is worth noting that unlike the argument in \cite{Mustapha:2017}, the error
analysis below does not need the regularity $\|u'(t)\|_{\dot H^2(\Omega)}$, which allows us to relax the regularity
assumption on the coefficient $a(x,t)$.
\end{remark}

\begin{remark}\label{rmk:high-reg}
Our discussions focus on the low regularity in space, i.e., $u(t)\in H^2(\Omega)$, which is
sufficient for the error analysis of the piecewise linear FEM in Section \ref{sec:semi}. These results
cannot be further improved for $u_0\in L^2(\Omega)$ or $f\in L^p(0,T;L^2(\Omega))$,
due to the limited smoothing properties of the solution operators {\rm(}at most of order two in space{\rm)}.
For smoother problem data, one may expect higher spatial regularity of the solution. For example, for the homogeneous
problem with a time-independent elliptic operator, there holds for any $\beta\geq 0$ \cite{SakamotoYamamoto:2011}
\begin{equation*}
 \|u(t)\|_{\dot H^{2+\beta}(\Omega)}\le ct^{-\alpha} \|u_0\|_{\dot H^\beta(\Omega)},\quad t>0.
\end{equation*}
Naturally, one may expect similar estimates for the case of a time-dependent elliptic operator, provided both
the domain $\Omega$ and the coefficient $a(x,t)$ are sufficiently smooth. Further, note that the
regularity analysis extends straightforwardly to the slightly more general elliptic operators with the potential
and convective terms, provided that the coefficients in the lower-order terms have suitable regularity.
\end{remark}

\section{Semi-discrete Galerkin finite element method}\label{sec:semi}
In this part we investigate the semidiscrete Galerkin FEM.
Let $\mathcal{T}_h$ be a shape regular quasi-uniform triangulation of the domain $\Omega$ into simplicial elements, and $h$ be
the maximal diameter of the elements. Let $S_{h}\subset H_0^1(D)$ be the space of continuous piecewise linear functions over
the triangulation $\mathcal{T}_h$. Then we define the $L^2(\Omega)$ orthogonal projection $P_h:L^2(\Omega)\to S_h$ by
\begin{equation*}
  (P_h\varphi,\chi) = (\varphi,\chi)\quad \forall\,\varphi\in L^2(\Omega),\,\,\forall\,\chi\in S_h.
\end{equation*}
The operator $P_h$ satisfies the following error estimate
\begin{equation*}
  \|P_h\varphi-\varphi\|_{L^2(\Omega)} + h\|\nabla (P_h\varphi-\varphi)\|_{L^2(\Omega)}\leq ch^q\|\varphi\|_{H^q(\Omega)},\quad \varphi\in \dot H^q(\Omega), \,\, q=1,2.
\end{equation*}

The spatially semidiscrete FEM for problem \eqref{eqn:pde} reads: find $u_h(t)\in S_h$ such that
\begin{align}\label{eqn:fem}
\begin{aligned}
&(\Dal u_h(t) ,\chi)
+(a(\cdot,t)\nabla u_h(t),\nabla \chi) = (f(\cdot,t),\chi) ,\quad\forall \chi\in S_h,\,t\in(0,T],\quad\mbox{with }u_h(0)=P_hu_0 .
\end{aligned}
\end{align}
Then we define a time-dependent operator $A_h(t):S_h\rightarrow S_h$ by
\begin{align*}
(A_h(t)v_h,\chi)=(a(\cdot,t)\nabla v_h,\nabla \chi) ,\quad
\forall\, v_h,\chi\in S_h .
\end{align*}
Under condition \eqref{Cond-1}, $A_h(t):S_h\rightarrow S_h$ is bounded and
invertible on $S_h$, and problem \eqref{eqn:fem} can be rewritten as
\begin{align}\label{eqn:fem2}
\begin{aligned}
&\Dal u_h(t)+A_h(t)u_h(t)=P_hf(t),\quad\forall\, t\in(0,T] ,\quad \mbox{with }u_h(0)=P_hu_0 .
\end{aligned}
\end{align}

\subsection{Perturbation lemmas}\label{ssec:perturbation}
In this part we give two crucial perturbation results.
We need a time-dependent Ritz projection operator $R_h(t):H_0^1(\Omega)\to S_h$ defined by
\begin{equation}\label{eqn:Ritz}
  (a(\cdot,t)\nabla R_h(t)\varphi,\nabla \chi) = (a(\cdot,t)\nabla \varphi,\nabla \chi),\quad \forall \varphi\in H_0^1(\Omega),\chi\in S_h.
\end{equation}
The operator $R_h(t)$ satisfies the following approximation property \cite[p. 99]{LuskinRannacher:1982}:
\begin{equation}\label{eqn:Ritz-error}
  \|R_h(t)\varphi-\varphi\|_{L^2(\Omega)} + h\|\nabla(R_h(t)\varphi-\varphi)\|_{L^2(\Omega)}\leq ch^q\|\varphi\|_{H^q(\Omega)},\quad \varphi\in \dot H^q(\Omega), q=1,2.
\end{equation}

\begin{lemma}\label{lemma-fem-1}
Under conditions \eqref{Cond-1}-\eqref{Cond-2}, the following estimate holds:
\begin{align*}
\|(I-A_h(t)^{-1}A_h(s))v_h\|_{L^2(\Omega)}
\le c|t-s|\|v_h\|_{L^2(\Omega)} ,\quad\forall\, v_h\in S_h.
\end{align*}
\end{lemma}
\begin{proof}
For any given $v_h\in S_h$, let $\varphi_h=A_h(s)v_h$ and $w_h=A_h(t)^{-1}\varphi_h$. Then
\begin{align*}
(A_h(t)w_h,\chi)=(\varphi_h,\chi)=(A_h(s)v_h,\chi),\quad\forall\,\chi\in S_h,
\end{align*}
which implies
\begin{align*}
(a(\cdot,t)\nabla w_h,\nabla \chi)
=(a(\cdot,s)\nabla v_h,\nabla \chi),\quad\forall\,\chi\in S_h.
\end{align*}
Consequently,
\begin{align*}
(a(\cdot,t)\nabla (w_h-v_h),\nabla \chi)
=((a(\cdot,s)-a(\cdot,t))\nabla v_h,\nabla \chi),\quad\forall\,\chi\in S_h .
\end{align*}
Let $\phi\in H^1_0(\Omega)$ be the weak solution of the elliptic problem
\begin{align}\label{PDE-phi-fem}
(a(\cdot,t)\nabla \phi,\nabla \xi)
=((a(\cdot,s)-a(\cdot,t))\nabla v_h,\nabla \xi),\quad\forall\,\xi\in H^1_0(\Omega) .
\end{align}
By Lax-Milgram theorem, $\phi$ satisfies the following \textit{a priori} estimate:
\begin{align*}
\|\phi\|_{H^1(\Omega)}\le c\|(a(\cdot,s)-a(\cdot,t))\nabla v_h\|_{L^2(\Omega)}
\le c|t-s|\|v_h\|_{H^1(\Omega)}.
\end{align*}
Thus, with the Ritz projection $R_h(t)$, cf. \eqref{eqn:Ritz}, we have $w_h-v_h=R_h(t)\phi$.

By the error estimate \eqref{eqn:Ritz-error} and the inverse inequality,
\begin{align*}
\|w_h-v_h-\phi\|_{L^2(\Omega)} &\le
ch\|\phi\|_{H^1(\Omega)}
\le ch|t-s|\|v_h\|_{H^1(\Omega)}\\
&\le c|t-s|\|v_h\|_{L^2(\Omega)}.
\end{align*}
Thus, the triangle inequality implies
\begin{align}\label{fem-wh-vh}
\|w_h-v_h\|_{L^2(\Omega)}
\le c|t-s|\|v_h\|_{L^2(\Omega)} +\|\phi\|_{L^2(\Omega)} .
\end{align}
For any $\varphi\in L^2(\Omega)$, let $\xi\in \dot H^2(\Omega)$ be the solution of the elliptic problem
\begin{align*}
-\nabla\cdot(a(\cdot,t)\nabla \xi) =\varphi.
\end{align*}
Then $\|\xi\|_{H^2(\Omega)}\leq c\|\varphi\|_{L^2(\Omega)}$.
By substituting $\xi$ into \eqref{PDE-phi-fem}, we obtain
\begin{align*}
|(\phi,\varphi)|
&=|(a(\cdot,t)\nabla \phi,\nabla \xi)|\\
&=|((a(\cdot,s)-a(\cdot,t))\nabla v_h,\nabla \xi)|\\
&=|(v_h,\nabla \cdot (a(\cdot,s)-a(\cdot,t))\nabla \xi)|\\
&\le c |t-s|\|v_h\|_{L^2(\Omega)} \|\xi\|_{H^2(\Omega)}\\
&\le c |t-s|\|v_h\|_{L^2(\Omega)} \|\varphi\|_{L^2(\Omega)}.
\end{align*}
This implies (via duality)
\begin{align*}
\|\phi\|_{L^2(\Omega)} = \sup_{\varphi\in L^2(\Omega)}\frac{|(\phi,\varphi)|}{\|\varphi\|_{L^2(\Omega)}}
\le c |t-s|\|v_h\|_{L^2(\Omega)} .
\end{align*}
Substituting the last inequality back into \eqref{fem-wh-vh}, we deduce
\begin{equation*}
  \|w_h-v_h\|_{L^2(\Omega)}\le c|t-s|\|v_h\|_{L^2(\Omega)}.
\end{equation*}
This completes the proof of Lemma \ref{lemma-fem-1}.
\end{proof}

\begin{remark}\label{rem:lemma-fem-1}
Note that the semidiscrete operator $A_h(t)$ is self-adjoint. Then Lemma \ref{lemma-fem-1}
together with a duality argument yields
\begin{align*}
\|(I-A_h(s)A_h(t)^{-1})v_h\|_{L^2(\Omega)}
\le c|t-s|\|v_h\|_{L^2(\Omega)} ,\quad\forall\, v_h\in S_h.
\end{align*}
Consequently,
\begin{align*}
\|( A_h(t)-A_h(s) )v_h\|_{L^2(\Omega)} &\le \|(I-A_h(s)A_h(t)^{-1})A_h(t)v_h\|_{L^2(\Omega)}\\
&\le c|t-s|\|A_h(t) v_h\|_{L^2(\Omega)} .
\end{align*}
Further, the interpolation between $\beta=0,1$ yields
\begin{align*}
\|A_h^\beta(t) (I-A_h(t)^{-1}A_h(s))v_h\|_{L^2(\Omega)} \le c|t-s|\|A_h^\beta(t) v_h\|_{L^2(\Omega)} .
\end{align*}
\end{remark}

The following result is the continuous analogue of Lemma \ref{lemma-fem-1}, and it is
independent interest.
\begin{corollary}\label{cor:basic-est}
Under conditions \eqref{Cond-1}-\eqref{Cond-2}, the following estimate holds:
\begin{align*}
\|(I-A(t)^{-1}A(s))v\|_{L^2(\Omega)} \le c|t-s|\|v\|_{L^2(\Omega)} ,\quad\forall\, v \in H_0^1(\Omega).
\end{align*}
\end{corollary}
\begin{proof}
For any $v\in \dot H^2(\Omega)$, there holds $ A_h(s) R_h(s) v = P_h A(s) v $ \cite[equation (1.34), p. 11]{Thomee:2006}.
By the standard error estimates for Galerkin FEM,
\begin{align*}
  & \| A_h(t)^{-1}A_h(s) R_h(s) v  -  A(t)^{-1} A(s) v \|_{L^2(\Omega)} \\
=& \| (A_h(t)^{-1}P_h   -  A(t)^{-1}) A(s) v \|_{L^2(\Omega)}\\
 \le& ch^2 \| A(s) v  \|_{L^2(\Omega)}\leq ch^2\|v\|_{H^2(\Omega)}.
\end{align*}
Then by Lemma \ref{lemma-fem-1} and the triangle inequality, we deduce
\begin{align*}
\|(I-A(t)^{-1}A(s)) v\|_{L^2(\Omega)} &\leq \|(I -A_h(t)^{-1}A_h(s)  ) R_h(s)v\|_{L^2(\Omega)}  + ch^2 \|  v  \|_{H^2(\Omega)}\\
&\le c|t-s| \| R_h(s)v \|_{L^2(\Omega)} + ch^2 \|   v  \|_{H^2(\Omega)} \\
&\le c|t-s| \| v \|_{L^2(\Omega)}+c|t-s| \| R_h(s)v-v \|_{L^2(\Omega)} + ch^2 \|   v  \|_{H^2(\Omega)} \\
&\le c|t-s| \| v \|_{L^2(\Omega)}+(c|t-s| +c)h^2\| v \|_{H^2(\Omega)} ,\quad\forall\, v\in \dot H^2(\Omega).
\end{align*}
Then the assertion follows by letting $h \rightarrow 0$ and noting that the space $\dot H^2(\Omega)$ is dense in $H_0^1(\Omega)$.
\end{proof}

\begin{lemma}\label{lemma-fem-2}
Under conditions \eqref{Cond-1}-\eqref{Cond-2}, the following estimate holds:
\begin{align}
\| (R_h(t)-R_h(s))v\|_{L^2(\Omega)}
\le ch^2 |t-s| \|v\|_{H^2(\Omega)} ,\quad\forall\, v\in \dot H^2(\Omega).
\end{align}
\end{lemma}
\begin{proof}
By the definition of Ritz projection, cf. \eqref{eqn:Ritz},
the difference $\eta_h=R_h(t)v-R_h(s)v\in S_h$ satisfies
\begin{align*}
(a(\cdot,s)\nabla \eta_h,\nabla \chi)
=((a(\cdot,t)-a(\cdot,s))\nabla(v-R_h(t)v),\nabla \chi) ,
\quad\forall\,\chi\in S_h.
\end{align*}
Let $\eta\in H^1_0(\Omega)$ be the weak solution of the elliptic problem
\begin{align}\label{PDE-phi-fem2}
(a(\cdot,s)\nabla \eta,\nabla \xi)
=((a(\cdot,t)-a(\cdot,s))\nabla(v-R_h(t)v),\nabla \xi) ,\quad
\forall\,\xi\in H^1_0(\Omega).
\end{align}
By the definition of $R_h(s)$, cf. \eqref{eqn:Ritz}, $\eta_h=R_h(s)\eta$
and by the error estimate \eqref{eqn:Ritz-error}, there holds
\begin{align*}
\|\eta_h-\eta\|_{L^2(\Omega)} &\le ch \|\eta\|_{H^1(\Omega)}
\le ch\|(a(\cdot,t)-a(\cdot,s))\nabla(v-R_h(t)v)\|_{L^2(\Omega)}\\
&\le ch^2|t-s|\|v\|_{H^2(\Omega)}.
\end{align*}
The triangle inequality implies
\begin{align}\label{eqn:esti-eta}
\|\eta_h\|_{L^2(\Omega)}
\le ch^2|t-s|\|v\|_{H^2(\Omega)}
+\|\eta\|_{L^2(\Omega)} .
\end{align}
Next we use a duality argument to bound $\|\eta\|_{L^2(\Omega)}$.
For any $\varphi\in L^2(\Omega)$, let $\xi\in \dot H^2(\Omega)$ be the solution of the elliptic problem
\begin{align*}
-\nabla\cdot(a(\cdot,s)\nabla \xi) =\varphi .
\end{align*}
Upon substituting $\xi$ into \eqref{PDE-phi-fem2}, we obtain
\begin{align*}
|(\eta,\varphi)|
&=|(a(\cdot,s)\nabla \eta ,\nabla\xi)|
=|((a(\cdot,t)-a(\cdot,s))\nabla (v-R_h(t)v),\nabla \xi)| \nonumber  \\
&=|(v-R_h(t)v,\nabla \cdot ((a(\cdot,t)-a(\cdot,s))\nabla \xi)| \nonumber  \\
&\le c \|v-R_h(t)v\|_{L^2(\Omega)} |t-s| \|\xi\|_{H^2(\Omega)} \nonumber  \\
&\le ch^2 \|v\|_{H^2(\Omega)}  |t-s| \|\varphi\|_{L^2(\Omega)} ,
\end{align*}
which implies (via duality)
\begin{align*}
\|\eta\|_{L^2(\Omega)}
\le ch^2 \|v\|_{H^2(\Omega)}  |t-s|  .
\end{align*}
Substituting the above inequality into \eqref{eqn:esti-eta} yields Lemma \ref{lemma-fem-2}.
\end{proof}

\subsection{Semidiscrete scheme and error estimates}

By the discrete maximal $L^p$-regularity,
one can show the existence and uniqueness of a FEM solution $u_h(t)$.
We also have the following stability estimates. The proof is identical with that for Theorems
\ref{thm:reg-space1}--\ref{thm:reg-time}, using the estimates in
Section \ref{ssec:perturbation}, and hence it is omitted.
\begin{theorem}\label{thm:reg-semi}
Let conditions \eqref{Cond-1}-\eqref{Cond-2} be fulfilled, and $u_h$ be the solution to problem \eqref{eqn:fem}.
\begin{itemize}
\item[(i)] For $ u_0\in \dot H^\beta(\Omega)$, $0\leq \beta\leq 2$, and $f=0$, then
\begin{equation*}
 \|A_hu_h(t)\|_{L^2(\Omega)} \le ct^{-(1-\beta/2)\alpha} \|u_0\|_{\dot H^\beta(\Omega)}\quad\mbox{and}\quad    \| u_h'(t) \|_{L^2(\Omega)} \le c t^{-(1-\alpha\beta/2)} \| u_0  \|_{\dot H^\beta(\Omega)}. 
\end{equation*}
\item[(ii)] For $u_0=0$, $f\in C([0,T]; L^2(\Omega))$ and $\int_0^t  (t-s)^{\alpha-1}\| f'(s) \|_{L^2(\Omega)}\,\d s<\infty$, then
\begin{align*}
   &\|A_hu_h\|_{L^p(0,T;L^2(\Omega))}
+\|\Dal u_h\|_{L^p(0,T;L^2(\Omega))}
\le c\|f\|_{L^p(0,T;L^2(\Omega))},\\
   & \| u_h '(t) \|_{L^2(\Omega)} \le c t^{-(1-\alpha)}  \| f(0) \|_{L^2(\Omega)} +c\int_0^t (t-s)^{\alpha-1} \| f'(s)\|_{L^2(\Omega)} \,\d s,\\
 & \int_{0}^{t} (t-s)^{\alpha-1}\|A_hu_h(s)\|_{L^2(\Omega)}\d s \leq c \|f(0)\|_{L^2(\Omega)}+c\int_0^{t}(t-s)^{\alpha-1}\|f'(s)\|_{L^2(\Omega)}\d s.
\end{align*}
\item[(iii)] For $u_0=0$ and $f\in L^p(0,T;L^2(\Omega))$ with $p\in({2/\alpha},\infty)$, then
\begin{equation*}
  \|u_h(t)\|_{H^1(\Omega)}\leq c\|f\|_{L^p(0,t;L^2(\Omega))}.
\end{equation*}
\end{itemize}
\end{theorem}

Now we derive error estimates for the semidiscrete solution $u_h$.
Problem \eqref{eqn:fem} can be rewritten as
\begin{align*}
\begin{aligned}
&\Dal u_h(t) + A_h(t_0)u_h(t)
= P_hf(t)+ (A_h(t_0)-A_h(t))u_h(t) ,\quad t\in(0,T],\quad\mbox{with }u_h(0)=P_hu_0,
\end{aligned}
\end{align*}
whose solution is given by
\begin{align}\label{eqn:Sol-expr-uh}
u_h(t) =F_h(t;t_0)P_hu_0+ \int_0^t E_h(t-s; t_0)\Big(P_hf(s) +(A_h(t_0)-A_h(s))u_h(s)\Big) \d s ,
\end{align}
where the semidiscrete solution operators $F_h(t;t_0)$ and $E_h(t;t_0)$ are defined respectively by
\begin{align*}
F_h(t;t_0):=\frac{1}{2\pi {\rm i}}\int_{\Gamma_{\theta,\delta }}e^{zt} z^{\alpha-1} (z^\alpha+A_h(t_0) )^{-1}\, \d z \quad \mbox{and}\quad
E_h(t;t_0):=\frac{1}{2\pi {\rm i}}\int_{\Gamma_{\theta,\delta}}e^{zt}  (z^\alpha+A_h(t_0) )^{-1}\, \d z.
\end{align*}

Let $e_h=P_hu-u_h$. Then by \eqref{eqn:Sol-expr-u} and \eqref{eqn:Sol-expr-uh}, $e_h$ can be represented by
\begin{align}\label{FE-Err-expr2}
e_h(t)=& (P_hF(t;t_0)u_0-F_h(t;t_0)P_hu_0)
+ \int_0^t (P_hE(t-s; t_0)-E_h(t-s; t_0)P_h)f(s)  \d s \nonumber \\
&+ \int_0^t (P_hE(t-s; t_0)-E_h(t-s; t_0)P_h)(A(t_0)-A(s))u(s) \d s \nonumber \\
&+ \int_0^t E_h(t-s; t_0)\Big(P_h(A(t_0)-A(s))u(s) -(A_h(t_0)-A_h(s))u_h(s)\Big) \d s \nonumber \\
=&\!: \sum_{i=1}^4{\rm I}_i(t).
\end{align}

The terms ${\rm I}_1(t)$ and ${\rm I}_2(t)$ represent the errors for
the homogeneous and inhomogeneous problems with a time-independent operator $A(t_0)$,
respectively, which have been analyzed: \cite[Theorem 3.7]{JinLazarovZhou:SIAM2013} implies
\begin{equation}\label{eqn:I1t0}
  \|{\rm I}_1(t_0)\|_{L^2(\Omega)}
   \le ct_0^{-(1-\beta/2)\alpha} h^2\|u_0\|_{\dot H^\beta(\Omega)},\quad \beta\in[0,2],
\end{equation}
and by the argument in \cite{JinLazarovPasciakZhou:2015}, there holds (with $\ell_h=\log (1+1/h)$)
\begin{equation}\label{eqn:err-space-rhs}
  \|{\rm I}_2(t_0)\|_{L^2(\Omega)} \le c h^2\ell_h^2\|f\|_{L^\infty(0,T;L^2(\Omega))}.
\end{equation}

It remains to bound the two terms ${\rm I}_3(t)$ and ${\rm I}_4(t)$, which are given below.
We shall discuss the homogeneous and inhomogeneous problems separately.
\begin{lemma}\label{lem:I3}
Under conditions \eqref{Cond-1} and \eqref{Cond-2}, for $u_0\in L^2(\Omega)$ and $f=0$, for the term ${\rm I}_3(t)$, there holds
\begin{equation*}
  \|{\rm I}_3(t_0)\|_{L^2(\Omega)} \leq ch^2 \|u_0\|_{L^2(\Omega)}.
\end{equation*}
\end{lemma}
\begin{proof}
By the definitions of the operators $E(t;t_0)$ and $E_h(t;t_0)$, we have
\begin{align*}
  \|{\rm I}_3(t_0)\|_{L^2(\Omega)}
\le& c\int_0^{t_0}\int_{\Gamma_{\theta,\delta}}|e^{z(t_0-s)}  |
\big\|(z^\alpha+A(t_0) )^{-1}-(z^\alpha+A_h(t_0) )^{-1}P_h\big\|_{L^2(\Omega)}
\|(A(t_0)-A(s))u(s) \|_{} \, |\d z |  \d s.
\end{align*}
By condition \eqref{Cond-1}, for any $z\in \Gamma_{\theta,\delta}$, we have \cite[p. 820]{FujitaSuzuki:1991}
\begin{equation*}
  \|(z^\alpha+A(t_0) )^{-1}-(z^\alpha+A_h(t_0) )^{-1}P_h\big\|_{} \leq ch^2,
\end{equation*}
where the constant $c$ is independent of $z$. Meanwhile, condition \eqref{Cond-2} implies
\begin{equation*}
  \|(A(t_0)-A(s))u(s) \|_{L^2(\Omega)} \leq c|t_0-s|\|u(s)\|_{H^2(\Omega)}.
\end{equation*}
Thus by Theorem \ref{thm:reg-space2},
\begin{align}
  \|{\rm I}_3(t_0)\|_{L^2(\Omega)} & \leq
ch^{2} \int_0^{t_0}(t_0-s)^{-1} (t_0-s) \|u(s)\|_{H^2(\Omega)} \d s\le ch^{2} \int_0^{t_0} \|u(s)\|_{H^{2}(\Omega)} \d s \nonumber\\
&= ch^{2} \int_0^{t_0} s^{-\alpha}\|u_0\|_{L^2(\Omega)} \d s
\le c h^{2}\|u_0\|_{L^2(\Omega)}.\label{eqn:est-I3last}
\end{align}
This completes the proof of the lemma.
\end{proof}

\begin{lemma}\label{lem:I4}
Under conditions \eqref{Cond-1} and \eqref{Cond-2}, for $u_0\in L^2(\Omega)$ and $f=0$, for the term ${\rm I}_4(t)$, there holds
\begin{equation*}
  \|{\rm I}_4(t_0)\|_{L^2(\Omega)}\leq ch^2\|u_0\|_{L^2(\Omega)} + c\int_0^{t_0} \|e_h(s)\|_{L^2(\Omega)}\d s.
\end{equation*}
\end{lemma}
\begin{proof}
Let $e_h=P_hu-u_h$. Using the identity $P_hA(s)=A_h(s)R_h(s)$ \cite[(1.34), p. 11]{Thomee:2006} and the triangle inequality, we derive
\begin{align*}
   \|{\rm I}_4(t_0)\|_{L^2(\Omega)}
&=\bigg\|\int_0^{t_0} E_h(t_0-s; t_0)\Big((A_h(t_0)R_h(t_0)-A_h(s)R_h(s))u(s) -(A_h(t_0)-A_h(s))u_h(s)\Big) \d s \bigg\|_{L^2(\Omega)} \\
&\le \bigg\|\int_0^{t_0} E_h(t_0-s; t_0)(A_h(t_0)-A_h(s))e_h(s)  \d s \bigg\|_{L^2(\Omega)} \\
&\quad +\bigg\|\int_0^{t_0} E_h(t_0-s; t_0)\Big(A_h(t_0)(R_h(t_0)-P_h)u(s) -A_h(s)(R_h(s)-P_h)u(s)\Big) \d s \bigg\|_{L^2(\Omega)} \\
&=: {\rm I}_{4,1}(t_0)+{\rm I}_{4,2}(t_0).
\end{align*}
For the term ${\rm I}_{4,1}(t_0)$, by Lemmas \ref{lem:smoothing}(ii) and \ref{lemma-fem-1}, we have
\begin{align*}
{\rm I}_{4,1}(t_0)&=\bigg\|\int_0^{t_0} A_h(t_0)E_h(t_0-s; t_0)(I-A_h(t_0)^{-1}A_h(s))e_h(s)  \d s \bigg\|_{L^2(\Omega)} \\
&\leq \int_0^{t_0}\|A_h(t_0)E_h(t_0-s; t_0)\| \|(I-A_h(t_0)^{-1}A_h(s))e_h(s)\|_{L^2(\Omega)} \d s\\
&\le c\int_0^{t_0} (t_0-s)^{-1}  (t_0-s) \|e_h(s)\|_{L^2(\Omega)}\d s= c\int_0^{t_0} \|e_h(s)\|_{L^2(\Omega)}\d s.
\end{align*}
For the term ${\rm I}_{4,2}(t_0)$, by the triangle inequality, we further split it into
\begin{align*}
{\rm I}_{4,2}(t_0) &\le
\bigg\|\int_0^{t_0} E_h(t_0-s; t_0) A_h(t_0)(R_h(t_0)-R_h(s))u(s) \d s \bigg\|_{L^2(\Omega)}  \\
&\quad+\bigg\| \int_0^{t_0} E_h(t_0-s; t_0)(A_h(t_0)-A_h(s))(R_h(s)-P_h)u(s)\Big) \d s \bigg\|_{L^2(\Omega)}
=:{\rm I}_{4,2}'(t_0)+{\rm I}_{4,2}''(t_0). 
\end{align*}
Now by Lemmas \ref{lem:smoothing}(ii) and \ref{lemma-fem-2} and Theorem \ref{thm:reg-space2}, we bound ${\rm I}_{4,2}'(t_0)$ by
\begin{align*}
 {\rm I}_{4,2}'(t_0) & \leq \int_0^{t_0}\|E_h(t_0-s;t_0)A_h(t_0)\|\|(R_h(t_0)-R_h(s))u(s)\|_{L^2(\Omega)}\d s\\
   &\leq c\int_0^{t_0}(t_0-s)^{-1}(t_0-s)h^2\|u(s)\|_{H^2(\Omega)}\d s \\
   & \le ch^2\int_0^{t_0}s^{-\alpha}\|u_0\|_{L^2(\Omega)} \d s \leq ch^2\|u_0\|_{L^2(\Omega)}.
\end{align*}
Likewise, by Lemma \ref{lemma-fem-1} and Theorem \ref{thm:reg-space2}, we bound $I_{4,2}''(t_0)$ by
\begin{align*}
  {\rm I}_{4,2}''(t_0) & = \bigg\| \int_0^{t_0} A_h(t_0)E_h(t_0-s; t_0)(I-A_h(t_0)^{-1}A_h(s))(R_h(s)-P_h)u(s)\Big) \d s \bigg\|_{L^2(\Omega)}\\
    & \leq  \int_0^{t_0}\| A_h(t_0)E_h(t_0-s; t_0)\|\|(I-A_h(t_0)^{-1}A_h(s))(R_h(s)-P_h)u(s)\|_{L^2(\Omega)} \d s \\
    &\leq c\int_0^{t_0} (t_0-s)^{-1}(t_0-s)\|(R_h(s)-P_h)u(s)\|_{L^2(\Omega)} \d s \\
    &\le ch^2\int_0^{t_0}\|u(s)\|_{H^2(\Omega)}  \d s \le ch^2\int_0^{t_0}s^{-\alpha}\|u_0\|_{L^2(\Omega)} \d s
    \le ch^2\|u_0\|_{L^2(\Omega)}.
\end{align*}
The desired assertion follows by combining the preceding estimates.
\end{proof}

Now we can state the main result of this part, i.e., error estimate on the
semidiscrete solution $u_h$.
\begin{theorem}\label{thm:err-space1}
Under conditions \eqref{Cond-1} and \eqref{Cond-2}, for $u_0\in L^2(\Omega)$ and $f=0$, there holds
\begin{equation*}
  \|u(t)-u_h(t)\|_{L^2(\Omega)}\leq ch^2t^{-\alpha}\|u_0\|_{L^2(\Omega)}.
\end{equation*}
\end{theorem}
\begin{proof}
Substituting \eqref{eqn:I1t0} and Lemmas \ref{lem:I3} and \ref{lem:I4} into \eqref{FE-Err-expr2} yields
\begin{equation*}
\|P_hu(t_0)-u_h(t_0)\|_{L^2(\Omega)}
\le c t_0^{-\alpha}  h^2\|u_0\|_{L^2(\Omega)}
+ c\int_0^{t_0} \|P_hu(s)-u_h(s)\|_{L^2(\Omega)}\d s,\quad\forall\, t_0\in(0,T].
\end{equation*}
By Gronwall's inequality from Lemma \ref{lem:gronwall}, we obtain
\begin{equation*}
\|P_hu(t)-u_h(t)\|_{L^2(\Omega)}
\le c t^{-\alpha}h^2\|u_0\|_{L^2(\Omega)} ,\quad\forall\, t\in(0,T].
\end{equation*}
By the approximation property of $P_h$ and Theorem \ref{thm:reg-space2}, we have
\begin{equation*}
  \|u(t_0)-P_hu(t_0)\|_{L^2(\Omega)}\leq ch^2\|u(t_0)\|_{H^2(\Omega)}\leq ct_0^{-\alpha}h^2\|u_0\|_{L^2(\Omega)}.
\end{equation*}
The last two estimates together imply the desired result.
\end{proof}

A similar error estimate holds for the inhomogeneous problem.
\begin{theorem}\label{thm:err-space2}
Under conditions \eqref{Cond-1} and \eqref{Cond-2}, for $u_0=0$ and $f\in L^\infty(0,T;L^2(\Omega))$, there holds
\begin{equation*}
  \|u(t)-u_h(t)\|_{L^2(\Omega)}\leq ch^2\ell_h^2\|f\|_{L^\infty(0,t;L^2(\Omega))},\quad \mbox{with } \ell_h = \log(1+1/h).
\end{equation*}
\end{theorem}
\begin{proof}
The proof is similar to Theorem \ref{thm:err-space1}, in view of \eqref{eqn:err-space-rhs}, and the following estimates:
\begin{align*}
  \|{\rm I}_3(t_0)\|_{L^2(\Omega)} & \leq ch^2 \|f\|_{L^\infty(0,t_0;L^2(\Omega))},\\
  \|{\rm I}_4(t_0)\|_{L^2(\Omega)}&\leq ch^2\|f\|_{L^\infty(0,t_0;L^2(\Omega))} + c\int_0^{t_0} \|e_h(s)\|_{L^2(\Omega)}\d s,
\end{align*}
which follow similarly as Lemmas \ref{lem:I3} and \ref{lem:I4}.
Actually, the first follows from \eqref{eqn:est-I3last} and Theorem \ref{thm:reg-space1} by
\begin{align*}
  \|{\rm I}_3(t_0)\|_{L^2(\Omega)}
 \le ch^{2} \int_0^{t_0} \|u(s)\|_{H^{2}(\Omega)} \d s \le ch^{2} \|f\|_{L^\infty(0,t_0;L^2(\Omega))}.
\end{align*}
Similarly, the second follows from the expressions of ${\rm I}_{4,1}'$ and ${\rm I}_{4,2}''$ in
Lemma \ref{lem:I4}, and Theorem \ref{thm:reg-space1}.
\end{proof}

\begin{remark}
We have only discussed discretization by piecewise linear finite elements. It is of much interest to
extend the analysis to high-order finite elements. This seems missing even for the case of a
time-independent diffusion coefficient when problem data are nonsmooth, partly due to the
limited smoothing property of the solution operators \cite{JinLazarovZhou:2018review}.
\end{remark}

\section{Time discretization}\label{sec:fully}
Now we study the time discretization of problem \eqref{eqn:pde}.
We divide the time interval $[0,T]$ into a uniform grid, with $ t_n=n\tau$, $n=0,\ldots,N$, and $\tau=T/N$ being
the time step size. Then we approximate the Riemann-Liouville fractional derivative
\begin{equation*}
^R\Dal \varphi(t)=\frac{1}{\Gamma(1-\alpha)}\frac{\d}{\d t}\int_0^t(t-s)^{-\alpha}\varphi(s)\d s
\end{equation*}
by the backward Euler (BE) convolution quadrature (with $\varphi^j=\varphi(t_j)$) \cite{Lubich:1986,JinLazarovZhou:SISC2016}:
\begin{equation*}
  ^R\Dal \varphi(t_n) \approx \tau^{-\alpha} \sum_{j=0}^nb_j\varphi^{n-j}:=\bar\partial_\tau^\alpha \varphi^n,\quad\mbox{ with } \sum_{j=0}^\infty b_j\xi^j = (1-\xi)^\alpha.
\end{equation*}

The fully discrete scheme for problem \eqref{eqn:pde} reads: find $u_h^n\in S_h$ such that
\begin{align}\label{eqn:fully}
\bDal (u_h^n-u_h^0)+A_h(t_n)u_h^n=P_hf(t_n),\quad n=1,2,\ldots,N,
\end{align}
with the initial condition $u_h^0=P_hu_0\in S_h$. Similar to the semidiscrete case,
for a given $m\in\mathbb{N}$ with $1\le m\le N$, we rewrite \eqref{eqn:fully} as
\begin{align}\label{eqn:CQ2}
\bDal (u_h^n-u_h^0)+A_h(t_m)u_h^n=P_hf( t_n) +(A_h(t_m)-A_h(t_n))u_h^n .
\end{align}
By means of discrete Laplace transform, the fully discrete solution $ u_h^m\in S_h$ is given by
\begin{align}\label{eqn:Sol-expr-uhtau}
  u_h^m = F_{\tau,m}^mu_h^0 + \tau \sum_{k=1}^m E_{\tau,m}^{m-k}[P_hf(t_k)+ (A_h(t_m)-A_h(t_k))u_h^k],
\end{align}
where the fully discrete operators $F_{\tau,m}^n$ and $E_{\tau,m}^n$ are respectively defined by (with $\delta_\tau(\xi)=(1-\xi)/\tau$)
\begin{align}
F_{\tau,m}^n &= \frac{1}{2\pi\mathrm{i}}\int_{\Gamma_{\theta,\delta}^\tau } e^{zn\tau} \delta_\tau(e^{-z\tau})^{\alpha-1}({ \delta_\tau(e^{-z\tau})^\alpha}+A_h(t_m))^{-1}\,\d z ,\\
E_{\tau,m}^n &= \frac{1}{2\pi\mathrm{i}}\int_{\Gamma_{\theta,\delta}^\tau } e^{zn\tau} ({ \delta_\tau(e^{-z\tau})^\alpha}+A_h(t_m))^{-1}\,\d z ,\label{op:disc}
\end{align}
with the contour
$\Gamma_{\theta,\delta}^\tau :=\{ z\in \Gamma_{\theta,\delta}:|\Im(z)|\le {\pi}/{\tau} \}$
(oriented with an increasing imaginary part).

The next lemma gives elementary properties of the kernel $\delta_\tau(e^{-z\tau})$. 
\begin{lemma}\label{lem:delta}
For any $\theta\in (\pi/2,\pi)$, there exists $\theta' \in (\pi/2,\pi)$  and
positive constants $c,c_1,c_2$ $($independent of $\tau$$)$ such that for all $z\in \Gamma_{\theta,\delta}^\tau$
\begin{equation*}
  \begin{aligned}
& c_1|z|\leq
|\delta_\tau(e^{-z\tau})|\leq c_2|z|,
&&\delta_\tau(e^{-z\tau})\in \Sigma_{\theta'}, \\
& |\delta_\tau(e^{-z\tau})-z|\le c\tau |z|^{2},
&& |\delta_\tau(e^{-z\tau})^\alpha-z^\alpha|\leq c\tau |z|^{1+\alpha}.
\end{aligned}
\end{equation*}
\end{lemma}

By the solution representations \eqref{eqn:Sol-expr-uh}
and \eqref{eqn:Sol-expr-uhtau}, the temporal error $e_h^m = u_h^m- u_h(t_m)$ satisfies
\begin{align}
  e_h^m &= (F_h(t_m;t_m)P_hu_0-F_{\tau,m}^m u_h^ 0 \big)
    + \big(\tau \sum_{k=1}^m E_{\tau,m}^{m-k}P_hf(t_k) - \sum_{k=1}^m\int_{t_{k-1}}^{t_k}E_h(t_m-s)P_hf(s)\d s\big)\nonumber\\
    &\quad + \big(\tau\sum_{k=1}^m E_{\tau,m}^{m-k}(A_h(t_m)-A_h(t_k))u_h^k - \sum_{k=1}^m\int_{t_{k-1}}^{t_k}E_h(t_m-s)(A_h(t_m)-A_h(s))u_h(s)ds\big)\nonumber\\
     & =\sum_{i=1}^3 {\rm I}_i^m.\label{eqn:fully-split}
\end{align}
For the first two terms, there hold \cite[Theorem 3.5]{JinLazarovZhou:SISC2016}
\begin{align*}
 \|{\rm I}_1^ m\|_{L^2(\Omega)} & \leq c\tau t_m^{-(1-\alpha\beta/2) } \| u_0 \|_{\dot H^\beta(\Omega)} ,\quad \beta\in [0,2],\\
 \|{\rm I}_2^m\|_{L^2(\Omega)} &\leq c\tau t_m^{-(1-\alpha)}\| f(0) \|_{L^2(\Omega)} +c \tau \int_0^{t_m} (t_m-s)^{\alpha-1} \| f'(s)  \|_{L^2(\Omega)} \,\d s.
\end{align*}

To estimate ${\rm I}_3^m$, we need two preliminary bounds on the operator $E_{\tau,m}^n$.
\begin{lemma}\label{lem:est-Etau}
For the operator $E_{\tau,m}^{m-k}$ defined in \eqref{op:disc}, there holds for any $\beta\in[0,1]$
\begin{equation*}
 \Big \| [\tau A_h^\beta(t_m) E_{\tau,m}^{m-k} -  \int_{t_{k-1}}^{t_k}A_h^\beta(t_m)E_h(t_m-s;t_m) \,\d s]\Big\|\leq c \tau^2(t_m-t_k+\tau)^{-(2-(1-\beta)\alpha)}.
\end{equation*}
\end{lemma}
\begin{proof}
First we consider the case $\beta=0$. By the definition of the operator $E_h(t;t_m)$, we have
\begin{align*}
 \int_{t_{k-1}}^{t_{k}} E_h(t_m-s;t_m) \,\d s
 &=  \frac{1}{2\pi {\rm i}}\int_{\Gamma_{\theta,\delta}}  (z^\alpha+A_h(t_m) )^{-1}\int_{t_{k-1}}^{t_{k}} e^{z(t_m-s)} \,\d s
\, \d z \\
 & = \frac{1}{2\pi {\rm i}}\int_{\Gamma_{\theta,\delta}}  e^{z(t_m-t_{k })} z^{-1}(e^{z\tau}-1)
 (z^\alpha+A_h(t_m) )^{-1}\, \d z.
\end{align*}
This and the defining relation \eqref{op:disc} yield
\begin{align*}
  & \tau E_{\tau,m}^{m-k}-\int_{t_{k-1}}^{t_{k}} E_h(t_m-s;t_m) \,\d s\\
  = & \frac{1}{2\pi {\rm i}}\int_{\Gamma_{\theta,\delta}^\tau } e^{z(t_m-t_k)}\left[ \tau({ \delta_\tau(e^{-z\tau})^\alpha}+A_h(t_m))^{-1}-z^{-1}(e^{z\tau}-1)
 (z^\alpha+A_h(t_m) )^{-1}\right]\,\d z\\
  & - \frac{1}{2\pi {\rm i}}\int_{\Gamma_{\theta,\delta}\setminus\Gamma_{\theta,\delta}^\tau}  e^{z(t_m-t_{k })} z^{-1}(e^{z\tau}-1)
 (z^\alpha+A_h(t_m) )^{-1}\, \d z:={\rm I} + {\rm II}.
\end{align*}
For $k<m$, let $\delta=(t_m-t_k+\tau)^{-1}$ and $z=s\cos\varphi+{\rm i}s\sin\varphi$.
By Lemma \ref{lem:delta} and \eqref{eqn:resol}, we obtain
\begin{align*}
\big\|  \tau({ \delta_\tau(e^{-z\tau})^\alpha}+A_h(t_m))^{-1}-z^{-1}(e^{z\tau}-1)
 (z^\alpha+A_h(t_m) )^{-1}  \big\| \le c \tau^2 |z|^{-\alpha+1},\quad \forall z \in \Gamma_{\theta,\delta}^\tau.
\end{align*}
Then the bound on the term ${\rm I}$ follows by
\begin{align*}
\| {\rm I} \| &  \le c \tau^2 \int_\delta^{\frac{\pi}{\tau\sin\theta}} e^{s(t_m-t_{k })\cos\theta} s^{-\alpha+1} \d s
+ c \tau^2\int_{-\theta}^\theta e^{\cos\varphi} \delta^{-\alpha+2}  \d \varphi   \le c\tau^2(t_m-t_k+\tau)^{\alpha-2} .
\end{align*}
Similarly, Taylor expansion of $e^{z\tau}$,
\eqref{eqn:resol} and Lemma \ref{lem:delta} bound the term ${\rm II}$ by
\begin{align*}
\| {\rm II} \|
& \le c\tau \int_{\Gamma_{\theta,\delta}\setminus\Gamma_{\theta,\delta}^\tau} |e^{z(t_m-t_{k })}||z|^{-\alpha}  |\d z|
\le c\tau \int_{\frac{\pi}{\tau\sin\theta}}^\infty  e^{s(t_m-t_{k })\cos\theta} s^{-\alpha}  \d s \\
&\le c  \tau^2 \int_{\frac{\pi}{\tau\sin\theta}}^\infty  e^{s(t_m-t_{k })\cos\theta} s^{1-\alpha} \d s
\le c \tau^2 (t_m-t_{k })^{-(2-\alpha)}.
\end{align*}
For $k=m$, there hold
\begin{align*}
\| {\rm I} \| &  \le c \tau^2 \int_\delta^{\frac{\pi}{\tau\sin\theta}} s^{-\alpha+1} \d s
+ c \tau^2\int_{-\theta}^\theta \delta^{-\alpha+2}  \d \varphi   \le c\tau^\alpha,\\
\| {\rm II} \|
& \le c \int_{\Gamma_{\theta,\delta}\setminus\Gamma_{\theta,\delta}^\tau} |z|^{-\alpha-1}  |\d z|
 \le c \int_{\frac{\pi}{\tau\sin\theta}}^\infty   s^{-\alpha-1}  \d s \le c  \tau^\alpha.
\end{align*}
The proof for the case $\beta=1$ is analogous, and the intermediate case $\beta\in(0,1)$ follows by interpolation.
\end{proof}

The next result gives the smoothing property of the operator $E_{\tau,m}^n$.
\begin{lemma}\label{lem:smoothing2}
For the operator $E_{\tau,m}^n$ defined in \eqref{op:disc}, there holds
\begin{equation*}
  \|A_h(t_m)E_{\tau,m}^n\|\leq c(t_n+\tau)^{-1},\quad n=0,1,\dots,N.
\end{equation*}
\end{lemma}
\begin{proof}
Upon letting $\delta=(t_n+\tau)^{-1}$ in $\Gamma_{\theta,\delta}^\tau$ and
$z=s\cos\varphi+{\rm i}s\sin\varphi$, by
\eqref{eqn:resol} and Lemma \ref{lem:delta},
 we have
\begin{align*}
\| A_h(t_m)E_{\tau,m}^n \| &= \bigg\|\frac{1}{2\pi\mathrm{i}}\int_{\Gamma_{\theta,\delta}^\tau } e^{zt_n} A_h(t_m)({ \delta_\tau(e^{-z\tau})^\alpha}+A_h(t_m))^{-1}\,\d z \bigg\|\\
&\le c\int_{(t_n+\tau)^{-1}}^{\frac{\pi}{\tau\sin\theta}} e^{st_n\cos\theta} \d s
+ c\int_{-\theta}^\theta e^{\cos\varphi} (t_n+\tau)^{-1}  \d \varphi
\le c(t_n+\tau)^{-1} .
\end{align*}
This completes the proof of the lemma.
\end{proof}

Below we analyze the scheme \eqref{eqn:fully} for the homogeneous and inhomogeneous problems separately.

\subsection{Error estimate for the homogeneous problem}
First we analyze the homogeneous problem. It suffices to bound the term ${\rm I}_3^m$ in
the splitting \eqref{eqn:fully-split}.
\begin{lemma}\label{lem:I3m}
Under conditions \eqref{Cond-1}-\eqref{Cond-2}, for $u_0\in L^2(\Omega)$ and $f= 0$, there holds
\begin{equation*}
  \|{\rm I}_3^m\|_{L^2(\Omega)}\leq c \tau \log(1+t_m/\tau) t_m^{-1}\|u_0\|_{L^2(\Omega)} + c\tau \sum_{k=1}^m\|e_h^k\|_{L^2(\Omega)}.
\end{equation*}
\end{lemma}
\begin{proof}
Let $e_h^k = u_h^k- u_h(t_k)$, and $Q(t)=(A_h(t_m)-A_h(t))u_h(t)$. Then we split the summand of ${\rm I}_3^m$ into
\begin{align*}
&\tau E_{\tau,m}^{m-k} (A_h(t_m)-A_h(t_k))u_h^k - \int_{t_{k-1}}^{t_k} E_h(t_m-s;t_m) Q(s)\,\d s\\
= &\big(\tau E_{\tau,m}^{m-k} (A_h(t_m)-A_h(t_k))e_h^k \big) + (\tau E_{\tau,m}^{m-k} -  \int_{t_{k-1}}^{t_k} E_h(t_m-s;t_m) \,\d s) Q(t_k) \\
& + \int_{t_{k-1}}^{t_k} E_h(t_m-s;t_m) (Q(t_k) - Q(s))\,\d s
=: {\rm I}_k + {\rm II}_{k} + {\rm III}_{k}.
\end{align*}
It remains to bound the terms ${\rm I}_k$, ${\rm II}_k$ and ${\rm III}_k$. First, Lemmas \ref{lem:smoothing2} and
\ref{lemma-fem-1} bound the term $\|{\rm I}_k\|_{L^2(\Omega)}$ by:
\begin{align*}
\|{\rm I}_k\|_{L^2(\Omega)}&= \tau \|  A_h(t_m)E_{\tau,m}^{m-k} (I-A_h^{-1}(t_m)A_h(t_k))e_h^k \|_{L^2(\Omega)}\\
&\le c\tau (t_m-t_k+\tau)^{-1} \|  (I-A_h^{-1}(t_m)A_h(t_k))e_h^k \|_{L^2(\Omega)} \le c\tau  \|  e_h^k  \|_{L^2(\Omega)}.
\end{align*}
Second, by Lemma \ref{lem:est-Etau} (with $\beta=0$) and Remark \ref{rem:lemma-fem-1}, we bound the term ${\rm II}_k$ by
\begin{align*}
 \|{\rm II}_k \|_{L^2(\Omega)} & \leq \|\tau E_{\tau,m}^{m-k} -  \int_{t_{k-1}}^{t_k} E_h(t_m-s;t_m) \,\d s\|\| Q(t_k)\|_{L^2(\Omega)}\\
  &\le c \tau^2 (t_m-t_{k}+\tau)^{\alpha-1}\| A_h(t_m)u_h( t_k) \|_{L^2(\Omega)},
\end{align*}
and consequently, by Theorem \ref{thm:reg-semi}(i), we deduce
\begin{align*}
 \sum_{k=1}^{m} \|{\rm II}_k  \|_{L^2(\Omega)}& \le c\tau^2 \sum_{k=1}^{m}  (t_m-t_{k}+\tau)^{\alpha-1} \|A_h(t_m)u_h( t_k) \|_{L^2(\Omega)}\\
 &\le  c\tau^2 \| u_0\|_{L^2(\Omega)} \sum_{k=1}^{m}   (t_m-t_{k}+\tau)^{\alpha-1} t_k ^{-\alpha} \le c \tau \| u_0 \|_{L^2(\Omega)},
\end{align*}
where the last line follows from the inequality
\begin{equation*}
  \tau\sum_{k=1}^{m}   (t_m-t_{k}+\tau)^{\alpha-1} t_k ^{-\alpha} \leq c.
\end{equation*}
Last, for the third term ${\rm III}_k$,
with $k=1$, by Lemma \ref{lemma-fem-1}, we have
\begin{align*}
 \|{\rm III}_1\|_{L^2(\Omega)}
 &\le  \int_{0}^{\tau}\| E_h(t_m-s;t_m) Q(\tau) \|_{L^2(\Omega)} \,\d s +  \int_{0}^{\tau}\| E_h(t_m-s;t_m) Q(s) \|_{L^2(\Omega)} \,\d s \\
& = \int_{0}^{\tau}\| A_h(t_m)E_h(t_m-s;t_m) A_h(t_m)^{-1}Q(\tau) \|_{L^2(\Omega)} \,\d s \\
&\quad +  \int_{0}^{\tau}\| A_h(t_m) E_h(t_m-s;t_m)A_h(t_m)^{-1} Q(s) \|_{L^2(\Omega)} \,\d s \\
& \le c\int_{0}^{\tau} (t_m-s)^{-1} ((t_m-\tau) + (t_m-s)) \,\d s \| u_0 \|_{L^2(\Omega)} \le c\tau \| u_0\|_{L^2(\Omega)}.
\end{align*}
Meanwhile, for $k>1$, we further split the term ${\rm III}_k$ into
\begin{align*}
{\rm III}_k =& \int_{t_{k-1}}^{t_k}   E_h(t_m-s;t_m) \int_{t_k}^s   Q'(\xi) \,\d \xi\,\d s\\
 =& \int_{t_{k-1}}^{t_k}   E_h(t_m-s;t_m) \int_{t_k}^s  (A_h(t_m)-A_h(\xi)) u_h'(\xi) \,\d \xi\,\d s\\
 &- \int_{t_{k-1}}^{t_k}   E_h(t_m-s;t_m) \int_{t_k}^s  A_h'(\xi) u_h(\xi) \,\d \xi\,\d s
 =: {\rm III}_{k,1}+{\rm III}_{k,2}.
\end{align*}
By Lemmas \ref{lemma-fem-1} and \ref{lem:smoothing}(ii), the term ${\rm III}_{k,1}$ for any $k> 1$ can be bounded by
\begin{align*}
\| {\rm III}_{k,1}\|_{L^2(\Omega)}  &\leq \int_{t_{k-1}}^{t_k} \| A_h(t_m) E_h(t_m-s;t_m)\| \int^{t_k}_s \|(I-A_h(t_m)^{-1}A_h(\xi)) u_h'(\xi)\|_{L^2(\Omega)} \,\d \xi\,\d s\\
&\le c\int_{t_{k-1}}^{t_k} (t_m-s)^{-1} \int^{t_k}_s (t_m-\xi) \|  u_h'(\xi)\|_{L^2(\Omega)} \,\d \xi\,\d s \\
&\leq c\int_{t_{k-1}}^{t_k} (t_m-s)^{-1} \int^{t_k}_s (t_m-\xi) \xi^{-1} \|  u_0\|_{L^2(\Omega)} \,\d \xi\,\d s,
\end{align*}
where the last step is due to Theorem \ref{thm:reg-semi}(i). Now we note the elementary inequality
\begin{align*}
  \int_{t_{k-1}}^{t_k} (t_m-s)^{-1} \int^{t_k}_s (t_m-\xi) &\xi^{-1} \d \xi\d s = \int_{t_{k-1}}^{t_k}(t_m-\xi)\xi^{-1}\int_{t_{k-1}}^\xi(t_m-s)^{-1}\d s \d \xi \\
   &\leq \int_{t_{k-1}}^{t_k} (t_m-\xi)(t_m-\xi)^{-1}\tau\xi^{-1}\d\xi = \tau \int_{t_{k-1}}^{t_k}\xi^{-1}\d\xi.
\end{align*}
Consequently,
\begin{equation*}
  \| {\rm III}_{k,1}\|_{L^2(\Omega)}  \le c\tau \int_{t_{k-1}}^{t_k}\xi^{-1}\,\d\xi\|u_0\|_{L^2(\Omega)},
\end{equation*}
and
\begin{equation*}
  \sum_{k=2}^m\| {\rm III}_{k,1}\|_{L^2(\Omega)} \leq c\tau \int_{\tau}^{t_m} \xi^{ -1}\,\d\xi\|u_0\|_{L^2(\Omega)} =c\tau \log(t_m/\tau).
\end{equation*}
Similarly, by Theorem \ref{thm:reg-semi}(i), the term ${\rm III}_{k,2}$ for any $k>1$ is bounded by
\begin{align*}
\| {\rm III}_{k,2} \|_{L^2(\Omega)}  &= \int_{t_{k-1}}^{t_k} \| E_h(t_m-s;t_m)\| \int_s^{t_k} \|  A_h'(\xi)u_h(\xi)\|_{L^2(\Omega)} \,\d\xi\,\d s\\
&\le c\int_{t_{k-1}}^{t_k} (t_m-s)^{\alpha-1} \int^{t_k}_s\xi^{- \alpha}  \,\d\xi\,\d s\|u_0\|_{L^2(\Omega)}\\
& \le c\tau \int_{t_{k-1}}^{t_k} (t_m-s)^{\alpha-1}s^{-\alpha }\,\d s\|u_0\|_{L^2(\Omega)},
\end{align*}
where the last line follows from the inequality
$\int_s^{t_k} \xi^{-\alpha}\d \xi\leq s^{-\alpha}\int_s^{t_k}\d \xi\leq s^{-\alpha}\tau.$
Thus,
\begin{equation*}
  \sum_{k=1}^m\| {\rm III}_{k,2} \|_{L^2(\Omega)}
 \le c\tau \int_{0}^{t_m} (t_m-s)^{\alpha-1}s^{-\alpha }\,\d s\|u_0\|_{L^2(\Omega)}\leq c\|u_0\|_{L^2(\Omega)}.
\end{equation*}
Hence, there holds
\begin{align*}
\sum_{k=1}^m \|{\rm III}_k\|_{L^2(\Omega)} \le  c \tau(1+\log(t_m/\tau))\|u_0\|_{L^2(\Omega)}.
\end{align*}
Combining the preceding estimates completes the proof of the lemma.
\end{proof}

Now we can state an error estimate for the homogeneous problem.
\begin{theorem}\label{thm:err-time1}
Under conditions \eqref{Cond-1}-\eqref{Cond-2}, $u_0\in L^2(\Omega)$ and $f=0$, there holds
\begin{equation*}
  \|u_h^m-u_h(t_m)\|_{L^2(\Omega)} \leq c\tau t_m^{-1}\log (1+t_m/\tau))\|u_0\|_{L^2(\Omega)}.
\end{equation*}
\end{theorem}
\begin{proof}
It follows from Lemma \ref{lem:I3m} that
\begin{align*}
 \| e_h^m  \|_{L^2(\Omega)} \le c\tau( t_m^{-1 } + \log (1+t_m/\tau))\|u_0\|_{L^2(\Omega)} + c \tau \sum_{k=1}^m \| e_h^k \|_{L^2(\Omega)}.
\end{align*}
The desired estimate follows from a variant of the discrete Gronwall's inequality \cite[p. 258]{Thomee:2006}.
\end{proof}

\begin{remark}\label{rmk:err-time}
The logarithmic factor $\log(1+t_m/\tau)$ is also present for the BE method for standard parabolic problems with
a time-dependent diffusion coefficient \cite[Theorem 2, p. 95]{LuskinRannacher:1982}.
For $u_0\in \dot H^\beta(\Omega)$, $\beta\in(0,2]$, it may be improved:
\begin{equation*}
  \|u_h^m-u_h(t_m)\|_{L^2(\Omega)} \leq c\tau t_m^{-(1-\beta\alpha/2)}\|u_0\|_{\dot H^\beta(\Omega)}.
\end{equation*}
In fact, the argument of Lemma \ref{lem:I3m} {\rm(}together with Theorem \ref{thm:reg-semi}{\rm(}i{\rm))}
implies
\begin{align*}
\|{\rm III}_{k,1}  \|_{L^2(\Omega)}
&\le c\int_{t_{k-1}}^{t_k} (t_m-s)^{-1} \int_s^{t_k} (t_m-\xi) \|  u_h'(\xi)\|_{L^2(\Omega)} \,\d\xi\,\d s\\
&\le c\int_{t_{k-1}}^{t_k} (t_m-s)^{-1} \int_s^{t_k} (t_m-\xi)\xi^{-(1-\alpha\beta/2)}\|u_0\|_{\dot H^\beta(\Omega)} \,\d\xi\,\d s\\
&\le c\tau \int_{t_{k-1}}^{t_k}\xi^{-(1-\beta\alpha/2)}\,\d\xi \| u_0\|_{\dot H^\beta(\Omega)},
\end{align*}
and
\begin{align*}
\|{\rm III}_{k,2}  \|_{L^2(\Omega)}  &= \int_{t_{k-1}}^{t_k} \| E_h(t_m-s;t_m)\| \int_s^{t_k} \|  A_h'(\xi)u_h(\xi)\|_{L^2(\Omega)} \,\d\xi\,\d s\\
&\le c\int_{t_{k-1}}^{t_k} (t_m-s)^{\alpha-1} \int_s^{t_k}\xi^{-(1-\beta/2)\alpha}  \,\d\xi\,\d s\| u_0\|_{\dot H^\beta(\Omega)}\\
& \le c\tau \int_{t_{k-1}}^{t_k} (t_m-s)^{\alpha-1}s^{-(1-\beta/2)\alpha}\,\d s \| u_0\|_{\dot H^\beta(\Omega)}.
\end{align*}
Thus, for any $\beta\in (0,2]$, there holds
\begin{align*}
\sum_{k=1}^m \bigg\| \int_{t_{k-1}}^{t_k} E(t_m-s;t_m)   (Q(t_k) -  Q(s))\,\d s \bigg\|  &\le  c\tau \| u_0\|_{\dot H^\beta(\Omega)}.
\end{align*}
Then the desired estimate follows by repeating the argument for Theorem \ref{thm:err-time1}.
\end{remark}

\subsection{Error estimate for the inhomogeneous problem}

Now we give the temporal discretization error for the inhomogeneous problem.

\begin{theorem}\label{thm:err-time2}
Under conditions \eqref{Cond-1}-\eqref{Cond-2}, $u_0=0$, $f\in C([0,T];L^2(\Omega))$  and
$\int_0^t(t-s)^{\alpha-1}\|f'(s)\|_{L^2(\Omega)}\d s<\infty$ for any $0<t\leq T$, there holds
\begin{equation*}
  \|u_h^m-u_h(t_m)\|_{L^2(\Omega)} \leq c\tau \Big(t_m^{-(1-\alpha)}\| f(0) \|_{L^2(\Omega)} + \int_0^{t_m} (t_m-s)^{\alpha-1} \| f'(s)  \|_{L^2(\Omega)} \,\d s\Big).
\end{equation*}
\end{theorem}
\begin{proof}
The argument is similar to Theorem \ref{thm:err-time1}, and thus we only sketch the main steps.
It suffices to bound the terms ${\rm II}_k$ and ${\rm III}_k$ in Lemma \ref{lem:I3m}.
By Theorem \ref{thm:reg-semi}(iii),
Remark \ref{rem:lemma-fem-1} and Lemma \ref{lem:est-Etau} (with $\beta=1/2$), the following estimate holds:
\begin{align*}
\sum_{k=1}^{m} \|{\rm II}_k  \|_{L^2(\Omega)}
& \le c\tau^2 \sum_{k=1}^{m}  (t_m-t_{k}+\tau)^{\alpha/2-1} \|A_h^{1/2}(t_m)u_h( t_k) \|_{L^2(\Omega)}\\
   & \leq c\tau^2 \sum_{k=1}^{m}  (t_m-t_{k}+\tau)^{\alpha/2-1}  \|  f \|_{L^\infty(0,t_m;L^2(\Omega))}
   \leq c\tau\|f\|_{L^\infty(0,t_m;L^2(\Omega))}.
\end{align*}
Meanwhile, upon noting $t_m\leq T$, we have
\begin{align*}
\|f\|_{L^\infty(0,t_m;L^2(\Omega))} &\le \| f(0) \|_{L^2(\Omega)} + \int_0^{t_m} \| f'(s)\|_{L^2(\Omega)} \,\d s\\
&\le \| f(0) \|_{L^2(\Omega)} + c_T\int_0^{t_m} (t_m-s)^{\alpha-1}\| f'(s)\|_{L^2(\Omega)} \,\d s.
\end{align*}
Next, the two terms ${\rm III}_{k,1}$ and ${\rm III}_{k,2}$ can be bounded respectively by
\begin{align*}
\|{\rm III}_{k,1}  \|_{L^2(\Omega)}
&\le c\int_{t_{k-1}}^{t_k} (t_m-s)^{-1} \int_s^{t_k} (t_m-\xi)\|u_h'(\xi)\|_{L^2(\Omega)} \,\d\xi\,\d s\\
&= c\int_{t_{k-1}}^{t_k} (t_m-\xi) \|u_h'(\xi)\|_{L^2(\Omega)} \int_{t_{k-1}}^\xi(t_m-s)^{-1} \,\d s\,\d\xi\\
&\le c\tau\int_{t_{k-1}}^{t_k}\|u_h'(\xi)\|_{L^2(\Omega)}\d\xi,
\end{align*}
and
\begin{align*}
\|{\rm III}_{k,2}  \|_{L^2(\Omega)}  &\leq \int_{t_{k-1}}^{t_k} \| E_h(t_m-s;t_m)\| \int_s^{t_k} \|A_h'(\xi)u_h(\xi)\|_{L^2(\Omega)} \,\d\xi\,\d s\\
&\leq c\int_{t_{k-1}}^{t_k} (t_m-s)^{\alpha-1} \int_s^{t_k} \|A_h'(\xi)u_h(\xi)\|_{L^2(\Omega)}  \,\d\xi\,\d s\\
& =  c\int_{t_{k-1}}^{t_k} \|  A_h'(\xi)u_h(\xi)\|_{L^2(\Omega)} \int_{t_{k-1}}^\xi(t_m-s)^{\alpha-1} \,\d s \d \xi\\
& \le c\tau \int_{t_{k-1}}^{t_k} (t_m-\xi)^{\alpha-1}\|A_h'(\xi)u_h(\xi)\|_{L^2(\Omega)}\d\xi.
\end{align*}
Then by Theorem \ref{thm:reg-semi}(ii), we have
\begin{align*}
  \sum_{k=1}^m\|{\rm III}_{k,1}  \|_{L^2(\Omega)} &\le c\tau\int_{0}^{t_m}\|u_h'(s)\|_{L^2(\Omega)}\d s\\
    & \leq c\tau\Big(\int_0^{t_m}\|f(0)\|_{L^2(\Omega)}\d s+\int_0^{t_m}\int_0^s(s-\xi)^{\alpha-1}\|f'(\xi)\|_{L^2(\Omega)}\d \xi\d s\Big)\\
    &  = c\tau\Big(\int_0^{t_m}\|f(0)\|_{L^2(\Omega)}\d s+\int_0^{t_m}\|f'(\xi)\|_{L^2(\Omega)}\int_\xi^{t_m}(s-\xi)^{\alpha-1}\d s\d \xi\Big)\\
    & \leq c\tau\Big(\|f(0)\|_{L^2(\Omega)}+\int_0^{t_m}(t_m-\xi)^{\alpha-1}\|f'(\xi)\|_{L^2(\Omega)}\d \xi\Big),
\end{align*}
and similarly, by Theorem \ref{thm:reg-semi}(ii), we deduce
\begin{align*}
  \sum_{k=1}^m\|{\rm III}_{k,2}  \|_{L^2(\Omega)} & \le c\tau \int_{0}^{t_m} (t_m-s)^{\alpha-1}\|A_h'(s)u_h(s)\|_{L^2(\Omega)}\d s\\
   & \leq c\tau \Big(\|f(0)\|_{L^2(\Omega)}+\int_0^{t_m}(t_m-s)^{\alpha-1}\|f'(s)\|_{L^2(\Omega)}\d s\Big).
\end{align*}
These estimates together with discrete Gronwall's inequality complete the proof.
\end{proof}

\begin{remark}\label{rmk:L1}
The proof techniques in this section apply to other first-order methods, e.g.,
L1 scheme \cite{LinXu:2007}, and similar error estimates can also be derived for these methods.
\end{remark}

\begin{remark}
We briefly comment on the dependence of the constant $c$ in error estimates on the fractional order $\alpha$.
At a few occasions, it can blow up as $\alpha\to 1^-$; see e.g., ${\rm I}_3(t_0)$ in Lemma \ref{lem:I3},
${\rm I}_{4,2}(t_0)$ in Lemma \ref{lem:I4} and ${\rm III}_{k,2}$ in Lemma \ref{lem:I3m}. This phenomenon does
not fully agree with the results for the continuous model. Such a blowup phenomenon appears also in some
existing error analysis; see, e.g., \cite[eq. (2.2)]{Mustapha:2017} and \cite[Lemma 4.3]{StynesORiordanGracia:2017},
and it is of interest to further refine the estimates to fill in the gap.
\end{remark}

\section{Numerical results}\label{sec:numerics}
Now we present numerical examples to verify the theoretical results in Sections \ref{sec:semi} and
\ref{sec:fully}. We consider problem \eqref{eqn:pde} with a time-dependent elliptic operator $A(t)=-(2+
\cos(t))\Delta$ on the domain $\Omega=(0,1)$ and the following two sets of problem data:
\begin{itemize}
\item[(a)] $u_0(x)=x^{-1/4}\in H^{1/4-\epsilon}(\Omega)$ with $\epsilon\in(0,1/4)$ and $f\equiv0$.
\item[(b)] $u_0(x)=0$  and $f=e^{t}(1+\chi_{(0,1/2)}(x))$.
\end{itemize}
Unless otherwise specified, the final time $T$ is fixed at $T=1$.

We divide the domain $\Omega$ into $M$ subintervals of equal length $h=1/M$.
The numerical solutions are computed using the Galerkin FEM in space, and the backward Euler (BE)
CQ or L1 scheme in time. To evaluate the convergence, we compute the spatial error $e_s$ and
temporal error $e_t$, respectively, defined by
\begin{equation*}
 e_s(t_N) =  \|u_h(t_N)-u(t_N)\|_{L^2(\Omega)}
 \quad \mbox{and}\quad
  e_t(t_N) =  \|u_h^N-u_h(t_N)\|_{L^2(\Omega)}.
\end{equation*}
Since the exact solution is unavailable, we compute reference solutions on a finer mesh: for
the error $e_s$, we take the time step $\tau=1/10000$ and mesh size $h=1/1280$, and for the
error $e_t$, take $h=1/100$ and $\tau=1/10000$, unless otherwise specified.

First we examine the spatial convergence of the semidiscrete Galerkin scheme \eqref{eqn:fem2}.
The spatial errors for case (a) are shown in Table \ref{tab:a-space}, which indicates a steady
$O(h^2)$ rate for the semidiscrete scheme \eqref{eqn:fem2}, just as predicted by Theorem 
\ref{thm:err-space1}. The $O(h^2)$ rate holds for all three fractional orders and different 
terminal times. Since the initial data is nonsmooth, the spatial error $e_s(t_N)$ decreases with 
the time $t_N$, which is in good agreement with the regularity result in Theorem \ref{thm:reg-space2}. 
To further illuminate the precise dependence of the spatial error $e_s(t_N)$ on $t_N$, 
in Table \ref{tab:a-space-sing}, we present the error $e_s$ as the time $t_N\to0$ for case (a). 
By repeating the argument for Theorem \ref{thm:err-space1}, there holds $e_s(t_N) \leq ct_N^{-(2-\beta)
\alpha/2}h^2\|v\|_{\dot H^\beta(\Omega)}$, $0\leq \beta\le 2$. For case (a), this estimate predicts 
an exponent $7\alpha/8$ for the dependence on the time $t_N$, which gives the numbers shown in the 
bracket in Table \ref{tab:a-space-sing}. Table \ref{tab:a-space-sing} indicates that the empirical rate agrees
excellently with the predicted one, fully confirming the analysis. Similar observations hold also 
for the numerical results for the inhomogeneous problem in case (b), cf. Table \ref{tab:b-space}. 
These results fully support the error analysis of the semidiscrete scheme in Section \ref{sec:semi}.

\begin{table}[hbt!]
\caption{Spatial errors $e_s$ for example (a) with $\tau=1/10000$ and $h=1/M$.}\label{tab:a-space}
\vspace{-.3cm}{\setlength{\tabcolsep}{7pt}
	\centering
	\begin{tabular}{|c|l|ccccc|c|}
		\hline
		$T$ &\diagbox[width=3.5em]{$\alpha$}{$N$}  &$10$ &$20$ & $40$ & $80$ & $160$ &rate \\
		\hline
		     &   $0.25$      & 1.44e-5 & 3.62e-6 & 9.06e-7 & 2.26e-7 & 5.62e-8 & 2.00 (2.00)\\
		$1$    &   $0.50$      & 1.02e-5 & 2.56e-6 & 6.40e-7 & 1.60e-7 & 3.97e-8 & 2.00 (2.00)\\
		     &   $0.75$      & 5.18e-6 & 1.30e-6 & 3.25e-7 & 8.12e-8 & 2.02e-8 & 2.00 (2.00)\\
    \hline
    		     &   $0.25$      & 6.26e-5 & 1.57e-5 & 3.93e-6 & 9.80e-7 & 2.44e-7 & 2.01 (2.00)\\
		$10^{-3}$    &   $0.50$      & 2.12e-4 & 5.31e-5 & 1.33e-5 & 3.32e-6 & 8.24e-7 & 2.01 (2.00)\\
		         &   $0.75$      & 5.99e-4 & 1.50e-4 & 3.75e-5 & 9.36e-6 & 2.33e-6 & 2.01 (2.00)\\
    \hline
   \end{tabular}}
\end{table}

\begin{table}[hbt!]
\caption{Spatial errors $e_s$ for example (a) with $h=1/200$ and $N=10000$, at $T=10^{-k}$.}\label{tab:a-space-sing}
\vspace{-.3cm}{\setlength{\tabcolsep}{7pt}
	\centering
	\begin{tabular}{ |l|cccccc|c|}
		\hline
		  \diagbox[width=3.5em]{$\alpha$}{$k$}  & $2$ & $3$ &$4$ &$5$ &$6$ &$7$  &rate \\
		\hline
	0.25    &  2.40e-6 & 4.04e-6 & 6.58e-6 & 1.05e-5 & 1.65e-5 & 2.66e-5    & 0.21 (0.22)\\
	0.5    & 5.28e-6  & 1.31e-5 & 3.36e-5 & 9.02e-5 & 2.40e-4 & 6.40e-4    & 0.43 (0.44)\\
	0.75       & 9.95e-6 & 3.90e-5 & 1.70e-4 & 7.45e-4 & 3.26e-3  & 1.39e-2    & 0.64 (0.66)\\
	\hline
   \end{tabular}}
\end{table}

\begin{table}[hbt!]
\caption{Spatial errors $e_s$ for example (b) at $T=1$ with $\tau=1/10000$ and $h=1/M$.}\label{tab:b-space}
\vspace{-.3cm}{\setlength{\tabcolsep}{7pt}
	\centering
	\begin{tabular}{ |l|ccccc|c|}
		\hline
		   \diagbox[width=3.5em]{$\alpha$}{$M$}  &$10$ &$20$ & $40$ & $80$ & $160$ &rate \\
		\hline
		         $0.25$      & 2.03e-4 & 5.06e-5 & 1.27e-5 & 3.16e-6 & 7.85e-7 & 2.01 (2.00)\\
		      $0.50$      & 2.08e-4 & 5.19e-5 & 1.30e-5 & 3.24e-6 & 8.04e-7 & 2.01 (2.00)\\
		         $0.75$      & 2.13e-4 & 5.32e-5 & 1.33e-5 & 3.32e-6 & 8.25e-7 & 2.01 (2.00)\\
    \hline
   \end{tabular}}
\end{table}

Next we turn to the temporal convergence, and present numerical results for both BE and L1 schemes, cf. Remark
\ref{rmk:L1}. The temporal errors $e_t$ for case (a) at two time instances are given in Table \ref{tab:a-time},
which indicate an $O(\tau)$ convergence rate for both time stepping schemes. Further, the accuracy of both
schemes is largely comparable. The convergence is very steady for both schemes, and the convergence rate is 
independent of the fractional order $\alpha$ and the final time $t_N$ (so long as it is fixed). Further, it is 
observed that the error $e_t$ decreases with the time $t_N$. To show the dependence of the temporal error $e_t
(t_N)$ with the time $t_N$, in Table \ref{tab:a-time-sing}, we present $e_t(t_N)$ as the time $t_N$ tends to 
zero. In view of Remark \ref{rmk:err-time}, there holds $e_t(t_N)\leq c\tau t_N^{-(1-\beta\alpha/2)}\|u_0\|_{
\dot H^\beta(\Omega)}$, $0<\beta\leq 2$. This estimate predicts a decay $O(N^{-\alpha/8})$  for case (a), which agrees excellently with
the empirical rate (in the bracket) in Table \ref{tab:a-time-sing}, thereby confirming the sharpness of the
error estimate. These observations hold also for the inhomogeneous problem in case (b), cf. Table \ref{tab:b-time}.
These numerical results fully support the error analysis of the fully discrete scheme in Section \ref{sec:fully}.

\begin{table}[hbt!]
\caption{Temporal errors $e_t$ for example (a) with $h=1/100$ and $\tau =T/N$.}\label{tab:a-time}
\vspace{-.3cm}{\setlength{\tabcolsep}{7pt}
	\centering
	\begin{tabular}{|c|c|l|ccccc|c|}
		\hline
		$T$&method &\diagbox[width=3.5em]{$\alpha$}{$N$} &$100$ &$200$ & $400$ & $800$ & $1600$  &rate \\
		\hline
		&     &   $0.25$    & 5.43e-5 & 2.71e-5 & 1.35e-5 & 6.76e-6 & 3.38e-6    & 1.00 (1.00)\\
		&BE   &   $0.50$      & 9.49e-5 & 4.73e-5 & 2.36e-5 & 1.18e-5 & 5.90e-6   & 1.00 (1.00)\\
		 $1$ &    &   $0.75$     & 9.01e-5 & 4.49e-5 & 2.24e-5 & 1.12e-5 & 5.59e-6     & 1.00 (1.00)\\
    \cline{2-9}
		&     &   $0.25$      & 4.35e-5 & 2.17e-5 & 1.08e-5 & 5.41e-6 & 2.70e-6   & 1.00 (1.00)\\
		&L1   &   $0.50$      & 6.33e-5 & 3.15e-5 & 1.57e-5 & 7.84e-6 & 3.92e-6   & 1.00 (1.00)\\
		&     &   $0.75$     & 5.12e-5 & 2.54e-5 & 1.26e-5 & 6.29e-6 & 3.14e-6   & 1.01 (1.00)\\
	\hline
	        &     &   $0.25$      & 2.00e-4 & 9.99e-5 & 4.99e-5 & 2.49e-5 & 1.25e-5   & 1.00 (1.00)\\
		&BE   &   $0.50$      & 8.16e-4 & 4.08e-4 & 2.04e-4 & 1.02e-4 & 5.10e-5   & 1.00 (1.00)\\
		 $10^{-3}$ &    &   $0.75$      & 7.58e-4 & 3.79e-4 & 1.89e-4 & 9.46e-5 & 4.73e-5   & 1.00 (1.00)\\
    \cline{2-9}
		&     &   $0.25$      & 1.69e-4 & 8.43e-5 & 4.21e-5 & 2.10e-5 & 1.05e-5  & 1.00 (1.00)\\
		&L1   &   $0.50$     & 8.08e-4 & 3.99e-4 & 1.98e-4 & 9.84e-5 & 4.90e-5    & 1.01 (1.00)\\
		&     &   $0.75$       & 8.28e-4 & 4.11e-4 & 2.04e-4 & 1.02e-4 & 5.07e-5   & 1.01 (1.00)\\
	\hline
   \end{tabular}}
\end{table}

\begin{table}[hbt!]
\caption{Temporal errors $e_t$ for example (a) with $\alpha=0.5$, $h=10^{-3}$ and $N=5$, at $T=10^{-k}$.}\label{tab:a-time-sing}
\vspace{-.3cm}{\setlength{\tabcolsep}{7pt}
	\centering
	\begin{tabular}{|l|c|cccccc|c|}
		\hline
		  \diagbox[width=3.5em]{$\alpha$}{$k$}
		  &  & $3$ & $4$ &$5$ &$6$ &$7$ &$8$  &rate \\
		\hline
	0.5 &	BE    & 1.65e-2  & 1.06e-2 & 8.79e-3 & 7.45e-3 & 6.33e-3 & 5.39e-3    & 0.07 (0.06)\\
		 &   L1     & 1.91e-2 & 1.41e-2 & 1.06e-2 & 8.95e-3  & 7.59e-3  & 6.46e-3   & 0.07 (0.06)\\
	\hline
	0.8 &	BE      & 1.61e-2 & 1.23e-2 & 9.52e-3 & 7.44e-3 & 5.84e-3  & 4.56e-3    & 0.11 (0.10)\\
		&    L1     & 1.83e-2 & 1.40e-2 & 1.08e-2 & 8.46e-3  & 6.64e-3  & 5.19e-3  & 0.11 (0.10)\\
	\hline
   \end{tabular}}
\end{table}

\begin{table}[hbt!]
\caption{Temporal errors $e_t$ for example (b) at $T=1$ with $h=1/100$ and $\tau=T/N$.}\label{tab:b-time}
\vspace{-.3cm}{\setlength{\tabcolsep}{7pt}
	\centering
	\begin{tabular}{|c|l|ccccc|c|}
		\hline
		method &\diagbox[width=3.5em]{$\alpha$}{$N$} &$100$ &$200$ & $400$ & $800$ & $1600$  &rate \\
		\hline
		    &   $0.25$      & 3.26e-6 & 1.63e-6 & 8.15e-7 & 4.07e-7 & 2.04e-7   & 1.00 (1.00)\\
		BE   &   $0.50$      & 4.76e-6 & 2.37e-6 & 1.18e-6 & 5.92e-7 & 2.96e-7   & 1.00 (1.00)\\
		    &   $0.75$      & 2.76e-6 & 1.37e-6 & 6.84e-7 & 3.41e-7 & 1.71e-7   & 1.00 (1.00)\\
\hline
		    &   $0.25$      & 2.33e-6 & 1.17e-6 & 5.85e-7 & 2.93e-7 & 1.47e-7   & 1.00 (1.00)\\
		L1   &   $0.50$     & 3.25e-6 & 1.64e-6 & 8.29e-7 & 4.18e-7 & 2.10e-7   & 0.99 (1.00)\\
		    &   $0.75$      & 1.85e-6 & 9.89e-7 & 5.22e-7 & 2.72e-7 & 1.41e-7   & 0.94 (1.00)\\
	\hline
   \end{tabular}}
\end{table}

\section*{Acknowledgements}
The research of B. Li is partially supported by a Hong Kong RGC grant (Project No. 15300817), and that of  Z. Zhou by a start-up grant from
the Hong Kong Polytechnic University and  Hong Kong RGC grant No. 25300818.
\bibliographystyle{abbrv}
\bibliography{frac}
\end{document}